\documentclass[12pt]{article}
\usepackage[a4paper, total={6in,10in}]
{geometry}
\usepackage[utf8]{inputenc}
\usepackage{fullpage}
\usepackage{amsmath}
\usepackage{amssymb}
\usepackage{amsfonts}
\usepackage{amsthm}
\usepackage{algorithm2e}
\usepackage{enumitem}
\newtheoremstyle{dotless}{}{}{\itshape}{}{\bfseries}{}{ }{}
\theoremstyle{dotless}
\newtheorem{theorem}{Theorem}[section]
\newtheorem{corollary}{Corollary}[section]
\newtheorem{lemma}{Lemma}[section]
\newtheorem{rem}{Remark}[section]

\usepackage{graphicx}

\DeclareMathOperator{\adj}{adj}

\title{ On the $A_{\alpha}$-spectra of some join graphs}

\author{Mainak Basunia\thanks{Department of Mathematics, Indian Institute of Technology Kharagpur, Kharagpur 721302, India. Email: leo28mynnix@gmail.com}\and Iswar Mahato\thanks{Department of Mathematics, Indian Institute of Technology Kharagpur, Kharagpur 721302, India. Email: iswarmahato02@gmail.com} \and M. Rajesh Kannan\thanks{Department of Mathematics, Indian Institute of Technology Kharagpur, Kharagpur 721302, India. Email: rajeshkannan@maths.iitkgp.ac.in, rajeshkannan1.m@gmail.com }}
\date{\today}
\begin{document}
\maketitle
\baselineskip=0.25in

\begin{abstract}\label{sec1}
Let $G$ be a simple, connected graph and let $A(G)$ be the adjacency matrix of $G$. If $D(G)$ is the diagonal matrix of the vertex degrees of $G$, then for every real $\alpha \in [0,1]$, the matrix $A_{\alpha}(G)$ is defined as
$$A_{\alpha}(G) = \alpha D(G) + (1- \alpha) A(G).$$
The eigenvalues of the matrix $A_{\alpha}(G)$ form the $A_{\alpha}$-spectrum of $G$. Let $G_1 \dot{\vee} G_2$, $G_1 \underline{\vee} G_2$, $G_1 \langle \textrm{v} \rangle G_2$ and $G_1 \langle \textrm{e} \rangle G_2$ denote the subdivision-vertex join, subdivision-edge join, $R$-vertex join and $R$-edge join of two graphs $G_1$ and $G_2$, respectively.
In this paper, we compute the $A_{\alpha}$-spectra of $G_1 \dot{\vee} G_2$, $G_1 \underline{\vee} G_2$, $G_1 \langle \textrm{v} \rangle G_2$ and $G_1 \langle \textrm{e} \rangle G_2$ for a regular graph $G_1$ and an arbitrary graph $G_2$ in terms of their $A_{\alpha}$-eigenvalues.  As an application of these results, we construct infinitely many pairs of $A_{\alpha}$-cospectral graphs.
\end{abstract}



\section{Introduction}\label{sec2}
All graphs considered in this article are simple, undirected and connected. Let $G=(V(G),E(G))$ be a graph with the vertex set $V(G)$ and the edge set $E(G)$. The adjacency matrix of $G$, denoted by $A(G)$, is an $n \times n$ symmetric matrix whose rows and columns are indexed by $V(G)$. The $(i,j)$-th entry of $A(G)$ is $1$, if the vertices $i$ and $j$ are adjacent in $G$, and 0 otherwise. We denote the degree of the vertex $v$ in $G$ by $d_{G}(v)$, and define $D(G)$ to be the $n \times n$ diagonal matrix, whose diagonal entries are the degrees of the vertices of $G$. The Laplacian matrix of $G$, denoted by $L(G)$, is defined as $L(G)=D(G)-A(G)$. The signless Laplacian matrix of $G$, denoted by $Q(G)$, is defined as $Q(G)=D(G)+A(G)$. In \cite{A_Q-merging_by_Nikiforov}, the author introduced 
 a family of  matrices $A_{\alpha}(G)$ as follows:
\begin{equation}\label{eq1}
A_{\alpha}(G) = \alpha D(G) + (1- \alpha) A(G),\qquad \text{for any }\alpha \in [0,1].
\end{equation}
It is clear that $A_{\alpha}(G)$ is equal to the adjacency matrix of $G$ if $\alpha=0$, and is equal to $\frac{1}{2}Q(G)$ if $\alpha=\frac{1}{2}$.

Given an $n \times n$ matrix $M$, let $M^T$, $\det (M)$ and $\adj(M)$ denote the transpose, the determinant  and the adjugate
of $M$, respectively. The characteristic polynomial of $M$ is denoted by $\psi_M(x)$, which is defined as
\begin{equation*}\label{eq2}
    \psi_M(x) = \det(x I_n -M),
\end{equation*}
where $I_n$ is the identity matrix of order $n$. In particular, for a graph $G$ on $n$ vertices, $\psi_{A(G)}(x)$ and $\psi_{A_{\alpha}(G)}(x)$ denote the characteristic polynomial of $A(G)$ and $A_{\alpha}(G)$, respectively. The roots of the characteristic polynomial of $M$ are called the $M$-eigenvalues. Let $\lambda_1(A(G)) \geq \lambda_2(A(G)) \geq \dots \geq \lambda_n(A(G))$ and $\lambda_1(A_{\alpha}(G)) \geq \lambda_2(A_{\alpha}(G)) \geq \dots \geq \lambda_n(A_{\alpha}(G))$ be the $A$-eigenvalues and $A_{\alpha}$-eigenvalues of $G$, respectively. The set of all eigenvalues of $A(G)$ and $A_{\alpha}(G)$  together with their multiplicities is called the $A$-spectrum and the $A_{\alpha}$-spectrum of $G$, respectively. If $\lambda_1>\lambda_2>\dots>\lambda_k$ are the distinct $A_{\alpha}$-eigenvalues of $G$, then the $A_{\alpha}$-spectrum of $G$  can be written as 
$$\sigma(A_{\alpha}(G))=\{[\lambda_1]^{m_1}, [\lambda_2]^{m_2}, \hdots ,[\lambda_k]^{m_k}\},$$
where $m_i$ is the algebraic multiplicity of $\lambda_i$, for $1\leq i\leq k$.
Two graphs are said to be $A$-cospectral (respectively, $A_{\alpha}$-cospectral) if they have the same $A$-spectrum (respectively, $A_{\alpha}$-spectrum).

In spectral graph theory, computing the spectra and the characteristic polynomials of various classes of matrices associated with the graphs are interesting problems considered in the literature. Various graph operations such as the disjoint union, the Cartesian product, the Kronecker product, the corona, the edge corona, the neighborhood corona, the subdivision-edge neighborhood corona, the join, the subdivision-vertex join, the subdivision-edge join, the $R$-vertex join, the $R$-edge join etc have been introduced and their adjacency, Laplacian and signless Laplacian spectra are computed in \cite{spec_of_corona_by_barik_pati_sarma, spec_of_gr_by_hamm,  spec_of_corona_by_Cui_&_Tian, intro_2_th_of_gr_spectra, spectra_of_gr_th&apps,spec_of_rvertjoin_&_redgejoin_by_das_&_panigrahi, spec_of_edge_corona_by_hou_shiu, spec_of_nbd_corona_by_gopalapillai, spec_of_gr_op_based_on_R_gr_by_Lan_Zhou,   spec_of_sub_vert_sub_edge_nbd_corona_by_Liu_Lu,spec_of_subvertexjoin_subedgejoin_by_Liu_&_Zhang, spec_of_nbd_corona_by_Liu_Zhou, spec_of_corona_by_Mcleman_&_Mcnicholas, signless_lap_spec_of_corona_and_edge_corona_by_Wang_Zhou}. Recently, the $A_{\alpha}$-spectra of some graph operations have been studied in \cite{A_alph_spec_of_graph_prod_by_Li_Wang, A_alpha_of_join_by_Lin_Liu_Xue, corona_gr_&_alpha_eigen_by_Tahir_&_Zhang}. Motivated by these works, in this article, we determine the $A_{\alpha}$-spectra of  subdivision-vertex join, subdivision-edge join, $R$-vertex join and $R$-edge join of two graphs $G_1$ and $G_2$, where $G_1$ is a regular graph and $G_2$ is an arbitrary graph. As applications of these results on the $A_{\alpha}$-spectra, we construct infinitely many pairs of $A_{\alpha}$-cospectral graphs. 
The results obtained in this paper extends the results  \cite{spec_of_rvertjoin_&_redgejoin_by_das_&_panigrahi,spec_of_subvertexjoin_subedgejoin_by_Liu_&_Zhang} for $A_\alpha$-spectra.

The join  of two graphs $G_1$ and $G_2$, denoted by $G_1 \vee G_2$, is the disjoint union of $G_1$ and $G_2$ together with all possible edges connecting all the vertices of $G_1$ with all the vertices of $G_2$  \cite{GT_by_Harary}. The subdivision graph  of a graph $G$, denoted by  $S(G)$, is the graph obtained by inserting a new vertex in every edge of $G$, that is by replacing each edge of $G$ by  $P_3$, the path on $3$ vertices \cite{intro_2_th_of_gr_spectra}. Based on this subdivision graph, two new graph operations, namely the subdivision-vertex join and the subdivision-edge join are introduced in \cite{spectra_of_two_new_joins_by_indulal}. The subdivision-vertex join of two graphs $G_1$ and $G_2$, denoted by $G_1 \dot{\vee} G_2$, is the graph obtained from $S(G_1)$ and $G_2$ by joining every vertex of $V(G_1)$ with every vertex of $V(G_2)$. The subdivision-edge join  of $G_1$ and $G_2$, denoted by $G_1 \underline{\vee} G_2$, is the graph obtained from $S(G_1)$ and $G_2$ by joining every vertex of $I(G_1)$ with every vertex of $V(G_2)$, where $I(G_1)$ is the set of inserted vertices of $S(G_1)$. The $R$-graph  of a graph $G$, denoted by $\mathcal{R}(G)$, is the graph obtained from $G$ by introducing a new vertex $u_e$ for each edge $e \in E(G)$, and making $u_e$ adjacent to both the end vertices of $e$ \cite{spectra_of_gr_th&apps}. In \cite{res_dist_&_kirchhoff_indx_by_Liu&Zhou&Bu}, the authors defined two new graph operations based on $R$-graph, namely the $R$-vertex join and the $R$-edge join. The $R$-vertex join of two graphs $G_1$ and $G_2$, denoted by $G_1\langle \textrm{v} \rangle G_2$, is the graph obtained from $\mathcal{R}(G_1)$ and $G_2$ by joining every vertex of $V(G_1)$ with every vertex of $V(G_2)$. The $R$-edge join of $G_1$ and $G_2$, denoted by $G_1 \langle \textrm{e} \rangle G_2$, is the graph obtained from $\mathcal{R}(G_1)$ and $G_2$ by joining every vertex of $I(G_1)$ with every vertex of $V(G_2)$,  where $I(G_1)$ is the set of inserted vertices of $\mathcal{R}(G_1)$.

This paper is organized as follows: In section \ref{sec3}, we collect some preliminary results and define some useful notations. In Section \ref{sec4}, \ref{sec5}, \ref{sec6} and \ref{sec7}, we obtain the characteristic polynomials of $A_{\alpha}$-matrices for the graphs $G_1 \dot{\vee} G_2$, $G_1 \underline{\vee} G_2$,  $G_1 \langle \textrm{v} \rangle G_2$ and $G_1 \langle \textrm{e} \rangle G_2$ respectively, where $G_1$ is an $r_1$-regular graph and $G_2$ is an arbitrary graph. In each of these four sections, we include some results on the eigenvalues of the said matrices taking $G_2$ as some particular graphs, like regular and complete bipartite. Also, as an application of these results, we construct infinitely many pairs of graphs having the same $A_{\alpha}$-spectrum.


\section{Preliminaries}\label{sec3}
Let $G$ be a graph on $n$ vertices and $m$ edges. The $\textit{incidence matrix}$ $R(G)$ of the graph $G$ is the $(0,1)$-matrix, whose rows and columns are indexed by the vertex set and the edge set of $G$, respectively. The $(i,j)$-th entry of $R(G)$ is $1$, if the vertex $i$ is incident to the edge $j$, and $0$  otherwise. The \textit{line graph} of $G$, denoted by $\mathcal{L}(G)$, is the graph with vertices are  the edges of $G$. Two vertices in $\mathcal{L}(G)$ are adjacent if and only if the corresponding edges have a common end-vertex in $G$. It is well known  \cite{intro_2_th_of_gr_spectra} that
\begin{equation}\label{eq3}
    R(G)^T R(G) = A(\mathcal{L}(G)) + 2I_m.
\end{equation}
If $G$ is an $r$-regular graph, then
\begin{equation}\label{eq4}
    R(G)R(G)^T = A(G) + rI_n.
\end{equation}

 We will use the symbols $\boldsymbol{0_n}$ and $\boldsymbol{1_n}$ ($\boldsymbol{0_{m\times n}}$ and $\boldsymbol{J_{m\times n}}$) for the column vectors ($m \times n$ matrices) consisting all $0$'s and all $1$'s, respectively. The $M$-coronal $\Gamma_M(x)$ of an $n \times n$ square matrix $M$ is defined \cite{spec_of_corona_by_Cui_&_Tian, spec_of_corona_by_Mcleman_&_Mcnicholas} by
\begin{equation}\label{eq5}
    \Gamma_M(x) = \boldsymbol{1_n^T}(x I_n - M)^{-1} \boldsymbol{1_n}.
\end{equation}
It is known that \cite[Proposition 2] {spec_of_corona_by_Cui_&_Tian}, if each row sum of an $n\times n$ matrix $M$ is constant, say $t$, then
\begin{equation}\label{eq6}
    \Gamma_M(x)=\frac{n}{x-t}.
\end{equation}

Now we state some lemmas which will be useful to prove our main results.
\begin{lemma} [\textbf{Schur complement formula}] \textup{\cite{schur_compl_&_its_appl_by_Zhang}} \label{lem1}
    Let $M_1$, $M_2$, $M_3$ and $M_4$ be matrices of size $r\times r$, $r\times s$, $s\times r$ and $s\times s$, respectively. Then
    \begin{equation*}\label{eq7}
        \begin{split}
            \det \begin{pmatrix}  M_1 & M_2\\M_3 & M_4 \end{pmatrix} &= \det(M_4) \cdot \det \Big(M_1 - M_2 M_4^{-1} M_3\Big), \textrm{when} \enskip M_4 \enskip \textrm{is invertible}.\\ & = \det(M_1) \cdot \det \Big(M_4 -M_3 M_1^{-1}M_2 \Big), \textrm{when} \enskip M_1 \enskip \textrm{is invertible}.
        \end{split}
    \end{equation*}
   The matrices $M_1 - M_2 M_4^{-1} M_3$ and $M_4 - M_3 M_1^{-1} M_2$ are called the \textit{Schur complements} of $M_4$ and $M_1$, respectively.
\end{lemma}

\begin{lemma}\textup{\cite[Proposition $2.2$]{spec_of_subvertexjoin_subedgejoin_by_Liu_&_Zhang}} \label{lem2}
    Let $A$ be an $n\times n$ real matrix. Then
    \begin{equation*}\label{eq8}
        \det(A+ cJ_{n\times n}) = \det(A) +  c\boldsymbol{1_n^T} \emph{adj}(A) \boldsymbol{1_n},
    \end{equation*}
    where $c$ is a real number.
\end{lemma}

\begin{lemma} \textup{\cite[Corollary $2.3$]{spec_of_subvertexjoin_subedgejoin_by_Liu_&_Zhang}} \label{lem3}
    Let $A$ be an $n\times n$ real matrix and $c$ be a real number. Then
    \begin{equation*} \label{eq9}
        \det(xI_n - A - c J_{n\times n}) = (1-c \Gamma_A(x)) \det (xI_n -A).
    \end{equation*}
\end{lemma}

\begin{lemma} \textup{\cite{intro_2_th_of_gr_spectra}} \label{lem10}
     Let $G$ be an $r$-regular graph on $n$ vertices, and let $\mathcal{L}(G)$ be the line graph of $G$. If the characteristic polynomials of the matrices $A(G)$ and $A(\mathcal{L}(G))$ are $\psi_{A(G)}(x)$ and $\psi_{A(\mathcal{L}(G))}(x)$, respectively, then
     $$ \psi_{A(\mathcal{L}(G))}(x) = (x+2)^{m-n} \psi_{A(G)}(x-r+2). $$
\end{lemma}

For two positive integers $p, q$, let $K_p$ and $K_{p,q}$ denote the complete graph on $p$ vertices, and complete bipartite graph on $p+ q$ vertices, respectively. 
\begin{lemma} \textup{\cite[Proposition $37$]{A_Q-merging_by_Nikiforov}} \label{lem4}
    The spectrum of $A_{\alpha}(K_{p,q})$ is\\
    $\sigma(A_{\alpha}(K_{p,q})) = \bigg\{ \frac{\alpha (p+q)+ \sqrt{{\alpha}^2 (p+q)^2 + 4pq(1-2\alpha)}}{2}$, $[\alpha p]^{q-1}$,$ [\alpha q]^{p-1}$, $\frac{\alpha (p+q)- \sqrt{{\alpha}^2 (p+q)^2 + 4pq(1-2\alpha)}}{2} \bigg\}$.
\end{lemma}

\begin{lemma} \textup{\cite[Theorem $3$]{corona_gr_&_alpha_eigen_by_Tahir_&_Zhang}} \label{lem6}
    For the graph $K_{p,q}$, the $A_{\alpha}$-coronal is given by,
    \begin{equation*} \label{eq10}
        \Gamma_{A_{\alpha}(K_{p,q})}(x) = \frac{(p+q)x
        -\alpha (p + q)^2 + 2pq}{x^2 -\alpha (p+q)x + (2\alpha -1)pq}.
    \end{equation*}
\end{lemma}
In \cite[Lemma 12]{spec_of_rvertjoin_&_redgejoin_by_das_&_panigrahi}, the authors have obtained an expression for the inverse of the matrix $(cI_n -dJ_{n\times n})$, for $c,d > 0$. We modify the conditions on $c$ and $d$, and restate the result with proof.
\begin{lemma} \label{lem5}
    Let $c$ and $d$ be two real numbers such that the matrix $(cI_n -dJ_{n\times n})$ is invertible. Then
    \begin{equation} \label{eq9_1}
        (cI_n -dJ_{n\times n})^{-1} = \frac{1}{c}I_n + \frac{d}{c(c-nd)}J_{n\times n}.
    \end{equation}
\end{lemma}

\begin{proof}
    The eigenvalues of $J_{n\times n}$ are $n$ and $0$, and $0$ has multiplicity $n-1$. Therefore, we have
    \begin{equation*}
        \det (cI_n -dJ_{n\times n}) = c^{n-1}(c-nd).
    \end{equation*}
    Since the matrix $(cI_n -dJ_{n\times n})$ is invertible, we have $ \det (cI_n -dJ_{n\times n}) = c^{n-1}(c-nd) \neq 0$. Thus $c \neq 0$ and $c-nd \neq 0$. So the expression on the right hand side of \eqref{eq9_1} is valid.
    Now,
    \begin{equation*}
    \begin{split}
        (cI_n -dJ_{n\times n}) \cdot \Big(\frac{1}{c}I_n + \frac{d}{c(c-nd)}J_{n\times n}\Big) &=I_{n} -\frac{d}{c}J_{n\times n} + \frac{d}{c-nd}J_{n\times n} -\frac{nd^2}{c(c-nd)}J_{n\times n}\\
        &=I_{n}.
    \end{split}
    \end{equation*}
    Hence the inverse of the matrix $(cI_n -dJ_{n\times n})$ is $\frac{1}{c}I_n + \frac{d}{c(c-nd)}J_{n\times n}$.
\end{proof}
\begin{rem}\label{labeling}
If $G_1$ is a graph on  $n_1$ vertices and $m_1$ edges, and $G_2$ is a graph on  $n_2$ vertices, each of the graphs  $G_1 \dot{\vee} G_2$, $G_1 \underline{\vee} G_2$,  $G_1 \langle \emph{\textrm{v}} \rangle G_2$ and $G_1 \langle \emph{\textrm{e}} \rangle G_2$ has $(n_1 + m_1 + n_2)$ vertices. We consider the following partition of the vertex set of above graphs:  $V(G_1) \cup I(G_1) \cup V(G_2)$, where $V(G_1)=\{v_1,v_2, \ldots, v_{n_1}\}$ and $V(G_2)=\{u_1,u_2, \ldots, u_{n_2}\}$ are vertex sets of $G_1$ and $G_2$, respectively, and $I(G_1) = \{v^\prime_1, v^\prime_2, \ldots, v^\prime_{m_1} \}$ is the set of inserted vertices to construct the graphs $S(G_1)$ and $R(G_1)$ from $G_1$. In the following figure, we illustrate the labeling process with a particular example.
\newpage
\begin{figure}[h]
\centering
\includegraphics[scale=.6]{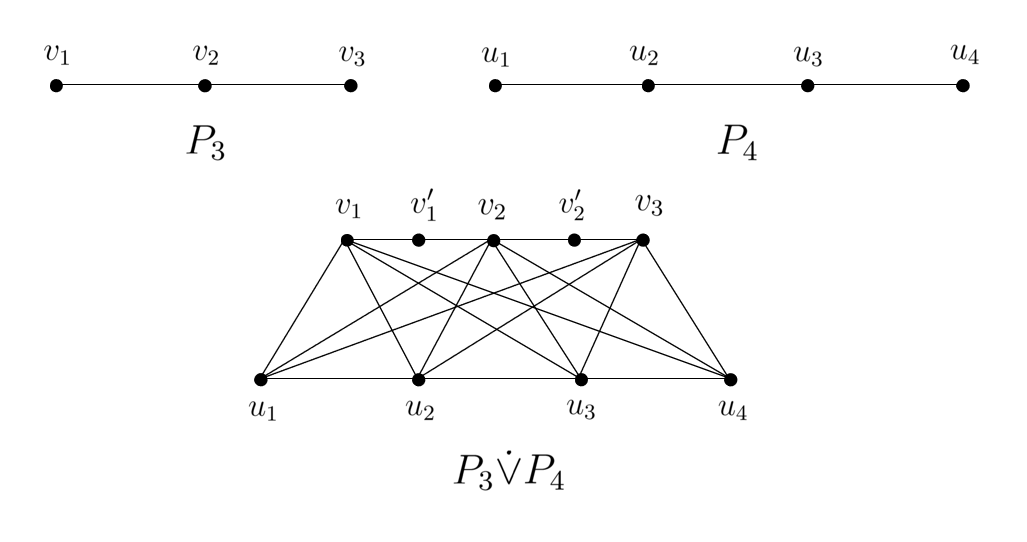}
\caption{The subdivision-vertex join of $P_3$ and $P_4$}
\end{figure}

\end{rem}

\section{$A_{\alpha}$-spectrum of $G_1 \dot{\vee} G_2$ } \label{sec4}
In this section, we discuss results related to the computation of  $A_{\alpha}$-spectrum of $G_1 \dot{\vee} G_2$, the subdivision-vertex join of the graphs $G_1$ and $G_2$. To begin with we obtain an expression for the $A_{\alpha}$-characteristic polynomial of $G_1 \dot{\vee} G_2$, where  $G_1$ is an $r_1$-regular graph and $G_2$ is an arbitrary graph.

\begin{theorem}\label{th1}
    Let $G_1$ be an $r_1$-regular graph on $n_1$ vertices and $m_1$ edges, and $G_2$ be a  graph on $n_2$ vertices. Let $\Gamma_{A_{\alpha}(G_2)}(x)$ be the $A_{\alpha}(G_2)$-coronal of $G_2$. Then, for each $\alpha \in [0,1]$, the $A_{\alpha}$-characteristic polynomial of $G_1 \dot{\vee} G_2$ is 
    \begin{multline} \label{eq11}
        \psi_{A_{\alpha}(G_1 \dot{\vee} G_2)}(x) =(x-2\alpha)^{m_1 -n_1} \cdot \psi_{A_{\alpha}(G_2)}(x- \alpha n_1)\\
        \cdot \prod_{i=2}^{n_1}\bigg(  x^2 -\alpha (2+ r_1 + n_2)x + \alpha (\alpha r_1 +r_1+2 \alpha n_2) -(1- \alpha)\Big( (1-\alpha)r_1 + \lambda_i\big(A_{\alpha}(G_1)\big) \Big)   \bigg)\\
        \cdot \bigg( x^2 -\alpha (2+r_1+n_2)x -2(r_1-2\alpha r_1 -\alpha^2 n_2) -n_1(1-\alpha)^2(x-2\alpha) \Gamma_{A_{\alpha}(G_2)}(x-\alpha n_1) \bigg).
    \end{multline}
\end{theorem}

\begin{proof}
With respect to the labeling of vertices considered in Remark \ref{labeling}, the adjacency matrix of $G_1 \dot{\vee} G_2$ is
\begin{equation}\label{eq12}
    A(G_1 \dot{\vee} G_2) =
    \begin{bmatrix}
        0_{n_1 \times n_1} & R & J_{n_1 \times n_2}\\
        R^T & 0_{m_1 \times m_1} & 0_{m_1 \times n_2}\\
        J_{n_2 \times n_1} & 0_{n_2 \times m_1} & A(G_2)
    \end{bmatrix},
\end{equation}
where $R$ is the $0-1$ incidence matrix of $G_1$.\\
The degrees of the vertices of the graph $G_1 \dot{\vee} G_2$ are:
\begin{equation*}\label{eq13}
    \begin{split}
        d_{G_1 \dot{\vee} G_2}(v_i) & = r_1 + n_2, \enskip \textrm{for} \enskip i=1,2, \ldots , n_1;\\
        d_{G_1 \dot{\vee} G_2}(v^\prime_j) & = 2, \enskip \textrm{for} \enskip j=1,2, \ldots , m_1;\\
        d_{G_1 \dot{\vee} G_2}(u_k) & = d_{G_2}(u_k) + n_1, \enskip \textrm{for} \enskip k=1,2, \ldots , n_2.
    \end{split}
\end{equation*}
So the diagonal matrix of order $(n_1 +m_1+n_2)\times (n_1+m_1+n_2)$, whose diagonal entries are the degrees of the vertices of the graph $G_1 \dot{\vee} G_2$ is
\begin{equation}\label{eq14}
    D(G_1 \dot{\vee} G_2) =
    \begin{bmatrix}
        (r_1 + n_2)I_{n_1} & 0_{n_1 \times m_1} & 0_{n_1 \times n_2}\\
        0_{m_1 \times n_1} & 2I_{m_1} & 0_{m_1 \times n_2}\\
        0_{n_2 \times n_1} & 0_{n_2 \times m_1} & D(G_2) + n_1I_{n_2}
    \end{bmatrix}.
\end{equation}
Using \eqref{eq12} and \eqref{eq14}, we have 
\begin{equation*}\label{eq15}
    A_{\alpha}(G_1 \dot{\vee} G_2)=
    \begin{bmatrix}
        \alpha(r_1 + n_2)I_{n_1} & (1-\alpha)R & (1-\alpha)J_{n_1 \times n_2}\\
        (1-\alpha)R^T & 2\alpha I_{m_1} & 0_{m_1 \times n_2}\\
        (1-\alpha)J_{n_2 \times n_1} & 0_{n_2 \times m_1} & A_{\alpha}(G_2) + \alpha n_1 I_{n_2}
    \end{bmatrix}.
\end{equation*}
Therefore, the characteristic polynomial of $A_{\alpha}(G_1 \dot{\vee} G_2)$ is
\begin{gather}\label{eq16}
    \begin{split}
    \psi_{A_{\alpha}(G_1 \dot{\vee}G_2)}(x) &= \det \Big(xI_{n_1+m_1+n_2}-A_{\alpha}(G_1 \dot{\vee} G_2) \Big)\\
    &= \det \begin{bmatrix}
        \Big(x-\alpha(r_1 + n_2)\Big) I_{n_1} & -(1-\alpha)R & -(1-\alpha)J_{n_1 \times n_2}\\
        -(1-\alpha)R^T & \Big(x-2\alpha \Big) I_{m_1} & 0_{m_1 \times n_2}\\
        -(1-\alpha)J_{n_2 \times n_1} & 0_{n_2 \times m_1} & (x-\alpha n_1) I_{n_2} - A_{\alpha}(G_2)
            \end{bmatrix}\\
    &=  \det \Big( (x-\alpha n_1) I_{n_2} - A_{\alpha}(G_2) \Big) \cdot \det S \quad (\text{by Lemma } \ref{lem1}),
    \end{split}
\end{gather}
where
\begin{multline*}
    S=
    \begin{bmatrix}
        \Big(x-\alpha(r_1 + n_2)\Big) I_{n_1} & -(1-\alpha)R\\
        -(1-\alpha)R^T & \Big(x-2\alpha \Big) I_{m_1}
    \end{bmatrix}\\
    -
    \begin{bmatrix}
        -(1-\alpha)J_{n_1 \times n_2}\\
        0_{m_1 \times n_2}
    \end{bmatrix}
    \Big( (x-\alpha n_1) I_{n_2} - A_{\alpha}(G_2) \Big)^{-1}
    \begin{bmatrix}
        -(1-\alpha)J_{n_2 \times n_1} & 0_{n_2 \times m_1}
    \end{bmatrix}.
\end{multline*}
\begin{multline*}
    \text{Now,   } \det S = \det \Bigg(
    \begin{bmatrix}
        \Big(x-\alpha(r_1 + n_2)\Big) I_{n_1} & -(1-\alpha)R\\
        -(1-\alpha)R^T & \Big(x-2\alpha \Big) I_{m_1}
    \end{bmatrix}\\
    - (1-\alpha)^2
    \begin{bmatrix}
        \Gamma_{A_{\alpha}(G_2)}(x-\alpha n_1) J_{n_1 \times n_1} & 0_{n_1 \times m_1}\\
        0_{m_1 \times n_1} & 0_{m_1 \times m_1}
    \end{bmatrix}
    \Bigg)
\end{multline*}
\begin{flalign*}
    = \det
    \begin{bmatrix}
        \Big(x-\alpha(r_1 + n_2)\Big) I_{n_1} - (1-\alpha)^2\Gamma_{A_{\alpha}(G_2)}(x-\alpha n_1) J_{n_1 \times n_1} & -(1-\alpha)R\\
        -(1-\alpha)R^T & \Big(x-2\alpha \Big) I_{m_1}
    \end{bmatrix}.    &&
\end{flalign*}
By Lemma \ref{lem1},  we have
\begin{multline*}
    \det S = \det \bigg( (x-2\alpha ) I_{m_1}\bigg) \\
    \cdot \det \bigg(  \Big(x-\alpha(r_1 + n_2)\Big) I_{n_1} - (1-\alpha)^2\Gamma_{A_{\alpha}(G_2)}(x-\alpha n_1)J_{n_1 \times n_1} \\
    -\Big((1-\alpha)R \Big) \Big( \big(x-2\alpha \big) I_{m_1} \Big)^{-1}\Big( (1-\alpha)R^T\Big) \bigg)
\end{multline*}

\begin{multline*}
    = (x-2\alpha)^{m_1}\\
    \cdot \det \bigg(  \Big(x-\alpha r_1 - \alpha n_2\Big) I_{n_1} - (1-\alpha)^2\Big(\Gamma_{A_{\alpha}(G_2)}(x-\alpha n_1) \Big)J_{n_1 \times n_1} -\frac{(1-\alpha)^2}{(x-2\alpha)}RR^T \bigg)
\end{multline*}
Now, by Lemma \ref{lem3}, we have
\begin{multline}\label{eq21}
    \det S= (x-2\alpha)^{m_1}  \cdot \det \bigg(  \big(x-\alpha r_1 - \alpha n_2\big) I_{n_1} - \frac{(1-\alpha)^2}{(x-2\alpha)}RR^T \bigg)\\
    \cdot \bigg( 1-  (1-\alpha)^2\Gamma_{A_{\alpha}(G_2)}(x-\alpha n_1) \Gamma_{\frac{(1-\alpha)^2}{(x-2\alpha)}RR^T}(x-\alpha r_1 -\alpha n_2)  \bigg).
\end{multline}

It is clear from equation \eqref{eq4} that the eigenvalues of the matrix $RR^T$ are $r_1 + \lambda_i(A(G_1))$ for $i=1,2, \ldots, n_1$. Again each row sum of the matrix $\frac{(1-\alpha)^2}{(x-2\alpha)} RR^T$ is  $\frac{2r_1(1-\alpha)^2}{(x-2\alpha)}$. Therefore from \eqref{eq6}, we have
\begin{equation*}\label{eq22}
    \Gamma_{\frac{(1-\alpha)^2}{(x-2\alpha)}RR^T}(x) = \frac{n_1}{x-\frac{2r_1(1-\alpha)^2}{(x-2\alpha)}} .
\end{equation*}
Thus \eqref{eq21} reduces to
\begin{multline*}
    \det S = (x-2\alpha)^{m_1}  \cdot \prod_{i=1}^{n_1} \bigg( x-\alpha r_1 - \alpha n_2 - \frac{(1-\alpha)^2}{(x-2\alpha)}\Big(r_1 +\lambda_i(A(G_1))\Big) \bigg)\\
    \cdot \bigg( 1-  (1-\alpha)^2 \frac{n_1(x-2\alpha)}{(x-2\alpha)(x-\alpha r_1 -\alpha n_2) -2r_1(1-\alpha)^2} \Gamma_{A_{\alpha}(G_2)}(x-\alpha n_1)  \bigg)
\end{multline*}
\begin{multline*}
    = (x-2\alpha)^{m_1-n_1}  \cdot \prod_{i=1}^{n_1} \bigg(\big(x-2\alpha\big) \big(x-\alpha r_1 - \alpha n_2\big) - (1-\alpha)^2\Big(r_1 +\lambda_i(A(G_1))\Big) \bigg)\\
    \cdot \frac{ (x-2\alpha)(x-\alpha r_1 -\alpha n_2) -2r_1(1-\alpha)^2  -  (1-\alpha)^2\Gamma_{A_{\alpha}(G_2)}(x-\alpha n_1) \cdot n_1(x-2\alpha)}{(x-2\alpha)(x-\alpha r_1 -\alpha n_2) -2r_1(1-\alpha)^2}
\end{multline*}
\begin{multline*}
    = (x-2\alpha)^{m_1-n_1}  \cdot \prod_{i=2}^{n_1} \bigg(\big(x-2\alpha\big) \big(x-\alpha r_1 - \alpha n_2\big) - (1-\alpha)^2\Big(r_1 +\lambda_i(A(G_1))\Big) \bigg)\\
    \cdot \bigg((x-2\alpha)(x-\alpha r_1 -\alpha n_2) -2r_1(1-\alpha)^2  -  n_1 (1-\alpha)^2(x-2\alpha)\cdot \Gamma_{A_{\alpha}(G_2)}(x-\alpha n_1) \bigg).
\end{multline*}
Since $A_{\alpha}(G_1) = \alpha r_1 I_{n_1} + (1-\alpha)A(G_1)$, the eigenvalues of $A_{\alpha}(G_1)$ are $\lambda_i(A_{\alpha}(G_1)) = \alpha r_1 + (1-\alpha)\lambda_i(A(G))$ for $i=1,2,\ldots,n_1$. Thus
\begin{multline*}
    \det S = (x-2\alpha)^{m_1-n_1} \\ \cdot \prod_{i=2}^{n_1} \bigg(\big(x-2\alpha\big) \big(x-\alpha r_1 - \alpha n_2\big) - (1-\alpha)^2 r_1 -(1-\alpha)\lambda_i\big(A_{\alpha}(G_1)\big) + \alpha (1-\alpha)r_1 \bigg)\\
    \cdot \bigg((x-2\alpha)(x-\alpha r_1 -\alpha n_2) -2r_1(1-\alpha)^2  -  n_1 (1-\alpha)^2(x-2\alpha)\cdot \Gamma_{A_{\alpha}(G_2)}(x-\alpha n_1) \bigg).
\end{multline*}
Simplifying this, we get the required result from \eqref{eq16}.
\end{proof}

Now, in the following corollary, we obtain the $A_{\alpha}$-eigenvalues of $G_1 \dot{\vee} G_2$, where  $G_2$ is an $r_2$-regular graph.

\begin{corollary}\label{cor1.1}
    Let $G_1$ be an $r_1$-regular graph on $n_1$ vertices and $m_1$ edges, and $G_2$ be an $r_2$-regular graph on $n_2$ vertices.
    \begin{enumerate}[label={\upshape \arabic*.}]
        \item If $r_1 =1$, then for each $\alpha \in [0,1]$, the $A_{\alpha}$-spectrum of $G_1 \dot{\vee} G_2$ consists precisely of:
        \begin{enumerate}[label= {\upshape (\roman*)}]
            \item $\alpha (1+n_2)$;
            \item $2\alpha + \lambda_i\big( A_{\alpha}(G_2) \big)$, $i=2,3, \ldots, n_2$ and
            \item three roots of the equation $F(x)=0$, where
            \begin{equation*}
                F(x)=(x-2\alpha -r_2)\big( x^2 -\alpha(3+n_2)x -2(1 -2\alpha - {\alpha}^2n_2) \big) - 2n_2(1-\alpha)^2(x-2\alpha).
            \end{equation*}
        \end{enumerate}

        \item If $r_1 \geq 2$, then for each $\alpha \in [0,1]$, the $A_{\alpha}$-spectrum of $G_1 \dot{\vee} G_2$ consists precisely of:
        \begin{enumerate}[label= {\upshape (\roman*)}]
            \item $2\alpha$, repeated $m_1 -n_1$ times;
            \item $\alpha n_1 + \lambda_i\big(A_{\alpha}(G_2)\big)$, $i=2,3, \ldots, n_2$;
            \item two roots of the equation $G_i(x) =0$ for each $i= 2,3,\ldots,n_1$, where
            \begin{multline*}
                G_i(x) = x^2 -\alpha (2+ r_1 +n_2)x + \alpha(\alpha r_1 +r_1 + 2\alpha n_2)
                \\-(1-\alpha)\big( r_1(1-\alpha) + \lambda_i\big(A_{\alpha}(G_1)\big) \big)
            \end{multline*}
            and
            \item three roots of the equation $F(x)=0$, where
            \begin{multline*}
                F(x)=(x-\alpha n_1 -r_2)\big( x^2 -\alpha(2+r_1+n_2)x -2(r_1 -2\alpha r_1 - {\alpha}^2n_2) \big)
                \\- n_1n_2(1-\alpha)^2(x-2\alpha).
            \end{multline*}
        \end{enumerate}
    \end{enumerate}
\end{corollary}

\begin{proof}
    Since $G_2$ is an $r_2$-regular graph on $n_2$ vertices, each row sum of the matrix $A_{\alpha}(G_2)$ is $r_2$. Therefore from \eqref{eq6}, we have
        \begin{equation}\label{eq27}
            \Gamma_{A_{\alpha}(G_2)}(x-\alpha n_1) = \frac{n_2}{x - \alpha n_1 -r_2}.
        \end{equation}
        Again $r_2$ is an eigenvalue of $A(G_2)$, therefore $r_2$ is an eigenvalue of $A_{\alpha}(G_2) = \alpha r_2 I_{n_2} + (1-\alpha)A(G_2)$.
        Using this and equation \eqref{eq27} in \eqref{eq11}, we get
        \begin{multline} \label{eq28}
            \psi_{A_{\alpha}(G_1 \dot{\vee} G_2)}(x) = (x-2\alpha)^{m_1-n_1} \prod_{i=2}^{n_2}\Big(x-\alpha n_1 -\lambda_i\big(A_{\alpha}(G_2)\big)\Big)\\
            \cdot \prod_{i=2}^{n_1}\Big( x^2 - \alpha \big( 2+r_1+n_2 \big)x + \alpha \big( \alpha r_1 +r_1 + 2\alpha n_2 \big) - \big( 1-\alpha \big) \big( r_1(1-\alpha) + \lambda_i(A_{\alpha}(G_1) \big) \big) \Big)\\
            \cdot \Big( \big( x-\alpha n_1 -r_2 \big) \big( x^2 -\alpha(2+r_1 +n_2)x -2(r_1 -2\alpha r_1 -\alpha^2 n_2) \big)  - n_1 n_2(1-\alpha)^2(x-2\alpha) \Big).
        \end{multline}
    \begin{enumerate}
        \item If $r_1 =1$, the only possibility for $G_1$ is  $P_2$, the path on two vertices. In this case, $n_1 =2$ and $m_1 =1$. Using these particular values of $r_1$, $n_1$ and $m_1$ in \eqref{eq28}, we get the desired result.

        \item If $r_1 \geq 2$, the graph $G_1$ can not have any pendant vertices, so it is not a tree. Thus $m_1 \geq n_1$, and the result follows from  \eqref{eq28}.
    \end{enumerate}
\end{proof}
Taking $G_2$ as $K_{p,q}$, we obtain the $A_{\alpha}$-eigenvalues of $G_1 \dot{\vee} G_2$ in the next corollary.

\begin{corollary}\label{cor1.2}
    Let $G_1$ be an $r_1$-regular graph with $n_1$ vertices and $m_1$ edges. Let $p,q \geq 1$ be integers and $G_2 = K_{p,q}$.
   \begin{enumerate}[label={\upshape \arabic*.}]
        \item If $r_1 =1$, then for each $\alpha \in [0,1]$, the $A_{\alpha}$-spectrum of $G_1 \dot{\vee} G_2$ consists precisely of:
        \begin{enumerate}[label= {\upshape (\roman*)}]
            \item $\alpha (1+p+q)$;
            \item $\alpha (p+2)$, repeated $q-1$ times;
            \item $\alpha (q+2)$, repeated $p-1$ times and
            \item four roots of the equation $F(x)=0$, where
            \begin{multline*}
                F(x) = \big( x^2 -\alpha(3+p+q)x - 2(1-2\alpha - {\alpha}^2p -{\alpha}^2q) \big) \\
                \cdot \big( x^2 -\alpha (4 + p +q)x +( 4{\alpha}^2 + 2{\alpha}^2p + 2{\alpha}^2q +2\alpha pq -pq) \big)\\
                -2(1-\alpha)^2(x-2\alpha) \big( (x-2\alpha)(p+q) -\alpha(p+q)^2 + 2pq \big).
            \end{multline*}
        \end{enumerate}

        \item If $r_1 \geq 2$, then for each $\alpha \in [0,1]$, the $A_{\alpha}$-spectrum of $G_1 \dot{\vee} G_2$ consists precisely of:
        \begin{enumerate}[label= {\upshape (\roman*)}]
            \item $2\alpha$, repeated $m_1 -n_1$ times;
            \item $\alpha (n_1 +p)$, repeated $q-1$ times;
            \item $\alpha (n_1 +q)$, repeated $p-1$ times;
            \item two roots of the equation $G_i(x) = 0$ for each $i=2,3,\ldots,n_1$, where
            \begin{multline*}
                G_i(x) = x^2 - \alpha (2+r_1+p+q)x + \alpha (\alpha r_1 + r_1 + 2\alpha p + 2\alpha q)\\
                -(1-\alpha)\big( (1-\alpha)r_1 + \lambda_i\big(A_{\alpha}(G_1)\big)\big)
            \end{multline*}
            and
            \item four roots of the equation $F(x)=0$, where
            \begin{multline*}
                F(x) = \big( x^2 -\alpha (2+r_1+p+q)x - 2 ( r_1-2\alpha r_1 - {\alpha}^2p -{\alpha}^2q)   \big) \\
                \cdot \big( x^2 -\alpha( 2n_1 + p +q)x +( {\alpha}^2n_1^2 + {\alpha}^2n_1p + {\alpha}^2n_1q +2\alpha pq -pq)  \big)\\
                -n_1(1-\alpha)^2(x-2\alpha)\big((x-\alpha n_1)(p+q) -\alpha( p+q )^2 +2pq \big).
            \end{multline*}
        \end{enumerate}
    \end{enumerate}
\end{corollary}

\begin{proof}
    From Lemma \ref{lem4} and Lemma \ref{lem6}, we have\\\\
    $\sigma(A_{\alpha}(G_2)) = \bigg\{ \frac{\alpha (p+q)+ \sqrt{{\alpha}^2 (p+q)^2 + 4pq(1-2\alpha)}}{2}$, $[\alpha p]^{q-1}$,$ [\alpha q]^{p-1}$, $\frac{\alpha (p+q)- \sqrt{{\alpha}^2 (p+q)^2 + 4pq(1-2\alpha)}}{2} \bigg\}$
    \\\\
    and
    \begin{equation*} \label{eq32}
        \Gamma_{A_{\alpha}(G_2)}(x) = \frac{(p+q)x -\alpha (p + q)^2 + 2pq}{x^2 -\alpha (p+q)x + (2\alpha -1)pq}.
    \end{equation*}
   Using this in \eqref{eq11}, we have
    
    \begin{multline} \label{eq33}
        \psi_{A_{\alpha}(G_1 \dot{\vee} G_2)}(x) = \big(x-2\alpha\big)^{m_1 -n_1} \\
        \cdot \Bigg(x-\alpha n_1 - \frac{\alpha(p+q) + \sqrt{\alpha^2 {(p+q)}^2 + 4pq(1-2\alpha)}}{2}\Bigg)
        \big(x-\alpha n_1 -\alpha p\big)^{q-1} \\
        \cdot \big(x-\alpha n_1 -\alpha q\big)^{p-1}
        \Bigg(x-\alpha n_1 - \frac{\alpha(p+q) - \sqrt{\alpha^2 {(p+q)}^2 + 4pq(1-2\alpha)}}{2}\Bigg) \\
        \cdot \prod_{i=2}^{n_1}\bigg(  x^2 -\alpha \big(2+ r_1 + p+q\big)x + \alpha \big(\alpha r_1 +r_1+2 \alpha(p+q)\big) -\big(1- \alpha\big)\Big( (1-\alpha)r_1 + \lambda_i\big(A_{\alpha}(G_1)\big) \Big) \bigg)\\
        \cdot \bigg( x^2 -\alpha \big(2+r_1+p+q\big)x -2\big(r_1-2\alpha r_1 -\alpha^2(p+q)\big) \\
        -n_1\big(1-\alpha\big)^2\big(x-2\alpha\big) \cdot \frac{(p+q)(x-\alpha n_1) -\alpha (p + q)^2 + 2pq}{(x-\alpha n_1)^2 -\alpha (p+q)(x-\alpha n_1) + (2\alpha -1)pq} \bigg).
    \end{multline}
    The zeros of the denominator of $\frac{(p+q)(x-\alpha n_1) -\alpha (p + q)^2 + 2pq}{(x-\alpha n_1)^2 -\alpha (p+q)(x-\alpha n_1) + (2\alpha -1)pq}$ are
    \begin{multline*}
        \alpha n_1 + \frac{\alpha(p+q) + \sqrt{\alpha^2 {(p+q)}^2 + 4pq(1-2\alpha)}}{2} \enskip \textrm{and}\\
        \alpha n_1 + \frac{\alpha(p+q) - \sqrt{\alpha^2 {(p+q)}^2 + 4pq(1-2\alpha)}}{2}.
    \end{multline*}
    Using this in \eqref{eq33},  we get
    \begin{multline}\label{eq35}
        \psi_{A_{\alpha}(G_1 \dot{\vee} G_2)}(x) = \big(x-2\alpha\big)^{m_1 -n_1} \big(x-\alpha n_1 -\alpha p\big)^{q-1}  \big(x-\alpha n_1 -\alpha q\big)^{p-1}\\
        \cdot \prod_{i=2}^{n_1}\bigg(  x^2 -\alpha \big(2+ r_1 + p+q\big)x + \alpha \big(\alpha r_1 +r_1+2 \alpha(p+q)\big) -\big(1- \alpha\big)\Big( (1-\alpha)r_1 + \lambda_i\big(A_{\alpha}(G_1)\big) \Big) \bigg)\\
        \cdot \Bigg(\Big( x^2 -\alpha \big(2+r_1+p+q\big)x -2\big(r_1-2\alpha r_1 -\alpha^2(p+q)\big)\Big)\\
        \cdot \Big( x^2 -(2\alpha n_1 + \alpha p + \alpha q)x +(\alpha^2 {n_1}^2 + \alpha^2 n_1 p + \alpha^2 n_1 q + 2\alpha pq -pq) \Big) \\
        -n_1\big(1-\alpha\big)^2\big(x-2\alpha\big) \Big((p+q)(x-\alpha n_1) -\alpha (p + q)^2 + 2pq \Big)\Bigg).
    \end{multline}
    \begin{enumerate}
        \item If $r_1 =1$, then $G_1=P_2$. Thus $n_1 =2$ and $m_1 =1$. Using these particular values of $r_1$, $n_1$ and $m_1$ in \eqref{eq35}, we get the desired result.

        \item If $r_1 \geq 2$, the graph $G_1$ is not a tree. Then we have $m_1 \geq n_1$, and the result follows from \eqref{eq35} .
    \end{enumerate}
\end{proof}

Finally, to conclude this section, we provide a construction of new pairs of $A_{\alpha}$-cospectral graphs from a given pair of $A_{\alpha}$-cospectral graphs in the following corollary. 

\begin{corollary}\label{cor1.3}
    \begin{enumerate}[label={\upshape \arabic*.}]
        \item Let $G_1$ and $G_2$ be two $A_{\alpha}$-cospectral regular graphs for $\alpha \in [0,1]$, and let $H$ be an arbitrary graph. Then the graphs $G_1 \dot{\vee} H$ and  $G_2 \dot{\vee} H$ are $A_{\alpha}$-cospectral.
        \item Let $H_1$ and $H_2$ be two $A_{\alpha}$-cospectral graphs with $\Gamma_{A_{\alpha}(H_1)}(x) = \Gamma_{A_{\alpha}(H_2)}(x)$ for $\alpha \in [0,1]$. If $G$ is a regular graph, then the graphs $G\dot{\vee} H_1$ and $G \dot{\vee} H_2$ are $A_{\alpha}$-cospectral.
    \end{enumerate}
\end{corollary}

\begin{proof}
    If two regular graphs are $A_{\alpha}$-cospectral, then they have same regularity with same number of vertices and  same number of edges. By applying  Theorem \ref{th1} on the concerned graphs and comparing their $A_{\alpha}$-characteristic polynomials, we get the required results.
\end{proof}


\section{$A_{\alpha}$-spectrum of $G_1 \underline{\vee} G_2$} \label{sec5}
This section is about the $A_{\alpha}$-spectrum of $G_1 \underline{\vee} G_2$, the subdivision-edge join of the graphs $G_1$ and $G_2$. We start by obtaining an expression for the $A_{\alpha}$-characteristic polynomial of $G_1 \underline{\vee} G_2$ for an $r_1$-regular graph $G_1$ and an arbitrary graph $G_2$.

\begin{theorem}\label{th2}
    Let $G_1$ be an $r_1$-regular graph on $n_1$ vertices and $m_1$ edges, and $G_2$ be a graph on $n_2$ vertices. Let $\Gamma_{A_{\alpha}(G_2)}(x)$ be the $A_{\alpha}(G_2)$-coronal of $G_2$.\begin{enumerate}[label={\upshape \arabic*.}]
        \item If $r_1 =1$, then for each $\alpha \in [0,1]$, the $A_{\alpha}$-characteristic polynomial of $G_1 \underline{\vee} G_2$ is given by,
        \begin{multline*}
            \psi_{A_{\alpha}(G_1 \underline{\vee} G_2)}(x) =(x-\alpha) \cdot \psi_{A_{\alpha}(G_2)}(x- \alpha)\\
            \cdot \Big( x^2 -\alpha (3+n_2)x +({\alpha}^2n_2 +4\alpha -2) - (1-\alpha)^2(x-\alpha) \Gamma_{A_{\alpha}(G_2)}(x-\alpha) \Big).
        \end{multline*}
        \item If $r_1 \geq 2$, then for each $\alpha \in [0,1]$, the $A_{\alpha}$-characteristic polynomial of $G_1 \underline{\vee} G_2$ is given by,
        \begin{multline*}
            \psi_{A_{\alpha}(G_1 \underline{\vee} G_2)}(x) =\big(x-2\alpha -\alpha n_2\big)^{m_1 -n_1} \cdot \psi_{A_{\alpha}(G_2)}(x- \alpha m_1)\\
            \cdot \Big( x^2 -\alpha (2+r_1+n_2)x +r_1({\alpha}^2n_2 +4\alpha -2) - m_1(1-\alpha)^2(x-\alpha r_1) \Gamma_{A_{\alpha}(G_2)}(x-\alpha m_1) \Big)\\
            \cdot \prod_{i=2}^{n_1}\Big(x^2 -\alpha (2+r_1+n_2)x + r_1({\alpha}^2n_2 +3\alpha -1) - (1-\alpha)\lambda_i\big(A_{\alpha}(G_1)\big)\Big).
         \end{multline*}
    \end{enumerate}
\end{theorem}

\begin{proof}
    Consider the labeling of the graph $G_1 \underline{\vee} G_2$ given in Remark \ref{labeling}. Then, the adjacency matrix of $G_1 \underline{\vee} G_2$ is
\begin{equation}\label{eq38}
    A(G_1 \underline{\vee} G_2) =
    \begin{bmatrix}
        0_{n_1 \times n_1} & R & 0_{n_1 \times n_2}\\
        R^T & 0_{m_1 \times m_1} & J_{m_1 \times n_2}\\
        0_{n_2 \times n_1} & J_{n_2 \times m_1} & A(G_2)
    \end{bmatrix},
\end{equation}
where $R$ is the $0-1$ incidence matrix of $G_1$.\\
The degrees of the vertices of the graph $G_1 \underline{\vee} G_2$ are:
\begin{equation*}\label{eq39}
    \begin{split}
        d_{G_1 \underline{\vee} G_2}(v_i) & = r_1, \enskip \textrm{for} \enskip i=1,2, \ldots , n_1;\\
        d_{G_1 \underline{\vee} G_2}(v^\prime_j) & = 2+n_2, \enskip \textrm{for} \enskip j=1,2, \ldots , m_1;\\
        d_{G_1 \underline{\vee} G_2}(u_k) & = d_{G_2}(u_k) + m_1, \enskip \textrm{for} \enskip k=1,2, \ldots , n_2.
    \end{split}
\end{equation*}
So the diagonal matrix with diagonal entries are the degrees of the vertices of the graph $G_1 \underline{\vee} G_2$ is
\begin{equation}\label{eq40}
    D(G_1 \underline{\vee} G_2) =
    \begin{bmatrix}
        r_1 I_{n_1} & 0_{n_1 \times m_1} & 0_{n_1 \times n_2}\\
        0_{m_ \times n_1} & (2+n_2) I_{m_1} & 0_{m_1 \times n_2}\\
        0_{n_2 \times n_1} & 0_{n_2 \times m_1} & D(G_2) + m_1 I_{n_2}
    \end{bmatrix}.
\end{equation}
Using \eqref{eq38} and \eqref{eq40}, we get the $A_{\alpha}$-matrix of $G_1 \underline{\vee} G_2$ as
\begin{equation*}\label{eq41}
    A_{\alpha}(G_1 \underline{\vee} G_2)=
    \begin{bmatrix}
        \alpha r_1 I_{n_1} & (1-\alpha)R & 0_{n_1 \times n_2}\\
        (1-\alpha)R^T & (2+n_2)\alpha I_{m_1} & (1-\alpha)J_{m_1 \times n_2}\\
        0_{n_2 \times n_1} & (1-\alpha)J_{n_2 \times m_1} & A_{\alpha}(G_2) + \alpha m_1 I_{n_2}
    \end{bmatrix}.
\end{equation*}
Therefore, the characteristic polynomial of $A_{\alpha}(G_1 \underline{\vee} G_2)$ is
\begin{gather}\label{eq42}
    \begin{split}
    \psi_{A_{\alpha}(G_1 \underline{\vee}G_2)}(x) &= \det \Big(xI_{n_1+m_1+n_2}-A_{\alpha}(G_1 \underline{\vee} G_2) \Big)\\
    &= \det \begin{bmatrix}
        \Big(x-\alpha r_1\Big) I_{n_1} & -(1-\alpha)R & 0_{n_1 \times n_2}\\
        -(1-\alpha)R^T & \Big(x-2\alpha -\alpha n_2 \Big) I_{m_1} & -(1-\alpha)J_{m_1 \times n_2}\\
        0_{n_2 \times n_1} & -(1-\alpha)J_{n_2 \times m_1} & (x-\alpha m_1) I_{n_2} - A_{\alpha}(G_2)
            \end{bmatrix}\\
    &=  \det \Big( (x-\alpha m_1) I_{n_2} - A_{\alpha}(G_2) \Big) \cdot \det S \quad (\text{by Lemma } \ref{lem1}),
    \end{split}
\end{gather}
where
\begin{multline*}
    S=
    \begin{bmatrix}
        (x-\alpha r_1) I_{n_1} & -(1-\alpha)R\\
        -(1-\alpha)R^T & (x-2\alpha-\alpha n_2) I_{m_1}
    \end{bmatrix}\\
    -
    \begin{bmatrix}
        0_{n_1 \times n_2}\\
        -(1-\alpha)J_{m_1 \times n_2}
    \end{bmatrix}
    \Big( (x-\alpha m_1) I_{n_2} - A_{\alpha}(G_2) \Big)^{-1}
    \begin{bmatrix}
        0_{n_2 \times n_1} & -(1-\alpha)J_{n_2 \times m_1}
    \end{bmatrix}
\end{multline*}

\begin{flalign*}
    =\begin{bmatrix}
        (x-\alpha r_1) I_{n_1} & -(1-\alpha)R\\
        -(1-\alpha)R^T & (x-2\alpha -\alpha n_2 ) I_{m_1} -(1-{\alpha})^2\Gamma_{A_{\alpha}(G_2)}\big(x-\alpha m_1\big)J_{m_1 \times m_1}
    \end{bmatrix}.    &&
\end{flalign*}
By  Lemma \ref{lem1}, we have
\begin{multline*}
    \det S = \det \Big( (x-\alpha r_1) I_{n_1}\Big) \\
    \cdot \det \Big( (x-2\alpha -\alpha n_2)I_{m_1} - (1-\alpha)^2\Gamma_{A_{\alpha}(G_2)}(x-\alpha m_1)J_{m_1 \times m_1} -\frac{(1-\alpha)^2}{(x-\alpha r_1)}R^TR \Big)
\end{multline*}
\begin{multline}\label{eq46}
    = (x-\alpha r_1)^{n_1}  \cdot \det \Big(  (x-2\alpha - \alpha n_2) I_{m_1} - \frac{(1-\alpha)^2}{(x-\alpha r_1)}R^TR \Big)\\
    \cdot \Big( 1-  (1-\alpha)^2\Gamma_{A_{\alpha}(G_2)}(x-\alpha m_1) \Gamma_{\frac{(1-\alpha)^2}{(x-\alpha r_1)}R^TR}(x-2\alpha -\alpha n_2)  \Big),\\ 
(\text{using Lemma } \ref{lem3}).
\end{multline}
The line graph of an $r$-regular graph is a $(2r-2)$-regular graph. Using this in \eqref{eq3} for the graph $G_1$, we get  each row sum of $R^TR$ is $2r_1$. Therefore using \eqref{eq6}, we get
\begin{equation*}\label{eq47}
    \Gamma_{\frac{(1-\alpha)^2}{(x-\alpha r_1)}R^TR}(x-2\alpha -\alpha n_2) = \frac{m_1}{x-2\alpha -\alpha n_2 -\frac{2r_1(1-\alpha)^2}{(x-\alpha r_1)}}.
\end{equation*}
Using these information, \eqref{eq46} becomes
\begin{multline}\label{eq48}
    \det S = (x-\alpha r_1)^{n_1}  \cdot \prod_{i=1}^{m_1} \Big( x-2\alpha- \alpha n_2 - \frac{(1-\alpha)^2}{(x-\alpha r_1)} \lambda_i\big(A\big(\mathcal{L}(G_1)\big) \big) -2\frac{(1-\alpha)^2}{(x-\alpha r_1)} \Big)\\
    \cdot \bigg( 1-  (1-\alpha)^2\Gamma_{A_{\alpha}(G_2)}(x-\alpha m_1) \cdot \frac{m_1(x-\alpha r_1)}{(x-\alpha r_1)(x-2\alpha -\alpha n_2) -2r_1(1-\alpha)^2}   \bigg),
\end{multline}
where $\mathcal{L}(G_1)$ is the line graph of $G_1$.
\begin{enumerate}[label={\upshape \arabic*.}]
    \item If $r_1=1$, then $G_1=P_2$. Therefore, $\mathcal{L}(G_1)$ is just a vertex and $0$ is the only eigenvalue of $A(\mathcal{L}(G_1))$.  Thus \eqref{eq48} becomes
    \begin{multline*}
        \det S = (x-\alpha)^2\cdot \Big(x-2\alpha -\alpha n_2 - 2\frac{(1-\alpha)^2}{(x-\alpha)}\Big)\\
        \cdot \bigg( 1- (1-\alpha)^2\Gamma_{A_{\alpha}(G_2)}(x-\alpha) \cdot \frac{(x-\alpha)}{(x-\alpha)(x-2\alpha -\alpha n_2) -2(1-\alpha)^2}  \bigg)
    \end{multline*}
    \begin{flalign*}
        = (x-\alpha)\Big((x-\alpha)(x-2\alpha -\alpha n_2) -2(1-\alpha)^2 -(1-\alpha)^2(x-\alpha)\Gamma_{A_{\alpha}(G_2)}(x-\alpha)\Big).
    \end{flalign*}
    After simplifying this, we get the required result from \eqref{eq42}.
    \item If $r_1 \geq 2$, then $m_1 \geq n_1$. By Lemma \ref{lem10} on $G_1$,   the eigenvalues of $A(\mathcal{L}(G_1))$ are $\lambda_i(A(G_1)) + r_1-2$ for $i=1,2, \ldots, n_1$ and $-2$, repeated for $(m_1 -n_1)$ times.
    Using these information, \eqref{eq48} becomes
    \begin{multline*}
        \det S = (x-\alpha r_1)^{n_1}\cdot \frac{1}{(x-\alpha r_1)^{m_1}}\\
        \cdot \prod_{i=1}^{n_1} \Big( (x-\alpha r_1)(x-2\alpha- \alpha n_2) - (1-\alpha)^2 \big(\lambda_i(A(G_1))+r_1 -2) \big) -2(1-\alpha)^2 \Big)\\
        \cdot \prod_{n_1 +1}^{m_1}  \Big( (x-\alpha r_1)(x-2\alpha- \alpha n_2) - (1-\alpha)^2(-2) - 2(1-\alpha)^2 \Big)\\
        \cdot \frac{ (x-\alpha r_1)(x-2\alpha -\alpha n_2) -2r_1(1-\alpha)^2-m_1(1-\alpha)^2 (x-\alpha r_1) \Gamma_{A_{\alpha}(G_2)}(x-\alpha m_1)}{(x-\alpha r_1)(x-2\alpha -\alpha n_2) -2r_1(1-\alpha)^2 }
    \end{multline*}
    \begin{multline}\label{eq52}
        = (x-2\alpha -\alpha n_2)^{m_1 -n_1}\\
        \cdot \prod_{i=2}^{n_1} \Big( (x-\alpha r_1)(x-2\alpha- \alpha n_2) - (1-\alpha)^2 \big(\lambda_i(A(G_1))+r_1 \big) \Big)\\
        \cdot
        \Big((x-\alpha r_1)(x-2\alpha -\alpha n_2) -2r_1(1-\alpha)^2-m_1(1-\alpha)^2 (x-\alpha r_1) \Gamma_{A_{\alpha}(G_2)}(x-\alpha m_1)\Big).
    \end{multline}
    Since $G_1$ is an $r_1$-regular graph, therefore
    \begin{equation*}\label{eq53}
        A_{\alpha}(G_1)=\alpha r_1 I_{n_1} +(1-\alpha)A(G_1).
    \end{equation*}
    Thus
    \begin{equation*}\label{eq54}
        \lambda_i\big(A_{\alpha}(G_1)\big)= \alpha r_1 + (1-\alpha)\lambda_i\big(A(G_1)\big) \enskip \textrm{for} \enskip i=1,2,\ldots, n_1,
    \end{equation*}
    and hence,
    \begin{equation}\label{eq55}
        \lambda_i\big(A(G_1)\big) = \frac{1}{(1-\alpha)}\Big(\lambda_i\big(A_{\alpha}(G_1)\big) - \alpha r_1\Big)\enskip \textrm{for} \enskip i=1,2,\ldots,n_1.
    \end{equation}
    Now  replace $\lambda_i\big(A(G_1)\big)$ in \eqref{eq52} using \eqref{eq55}. After simplifying that, we get the desired result from \eqref{eq42}.
\end{enumerate}
\end{proof}

Now, in the following corollary, we obtain the $A_{\alpha}$-eigenvalues of $G_1 \underline{\vee} G_2$ taking $G_2$ as an $r_2$-regular graph.

\begin{corollary}\label{cor2.1}
    Let $G_1$ be an $r_1$-regular graph on $n_1$ vertices and $m_1$ edges, and $G_2$ be an $r_2$-regular graph on $n_2$ vertices.
    \begin{enumerate}[label={\upshape \arabic*.}]
        \item If $r_1 =1$, then for each $\alpha \in [0,1]$, the $A_{\alpha}$-spectrum of $G_1 \underline{\vee} G_2$ consists precisely of:
        \begin{enumerate}[label= {\upshape (\roman*)}]
            \item $\alpha$;
            \item $\alpha + \lambda_i\big( A_{\alpha}(G_2) \big)$, $i=2,3, \ldots, n_2$ and
            \item three roots of the equation $F(x)=0$, where
            \begin{equation*}\label{eq56}
                F(x)=(x-\alpha -r_2)\big( x^2 -\alpha(3+n_2)x +({\alpha}^2n_2 +4\alpha -2) \big) - n_2(1-\alpha)^2(x-\alpha).
            \end{equation*}
        \end{enumerate}

        \item If $r_1 \geq 2$, then for each $\alpha \in [0,1]$, the $A_{\alpha}$-spectrum of $G_1 \underline{\vee} G_2$ consists precisely of:
        \begin{enumerate}[label= {\upshape (\roman*)}]
            \item $\alpha(2+n_2)$, repeated $m_1 -n_1$ times;
            \item $\alpha m_1 + \lambda_i\big(A_{\alpha}(G_2)\big)$, $i=2,3, \ldots, n_2$;
            \item two roots of the equation $G_i(x) =0$ for each $i= 2,3,\ldots,n_1$, where
            \begin{equation*}\label{eq57}
                G_i(x) = x^2 -\alpha(2+ r_1 +n_2)x + ({\alpha}^2 r_1 n_2 +3\alpha r_1 -r_1) - (1-\alpha) \lambda_i\big(A_{\alpha}(G_1)\big)
            \end{equation*}
            and
            \item three roots of the equation $F(x)=0$, where
            \begin{multline*}
                F(x)=(x-\alpha m_1 -r_2)\big( x^2 -\alpha(2+r_1+n_2)x +r_1({\alpha}^2n_2 +4\alpha -2) \big)\\
                -m_1n_2(1-\alpha)^2(x-\alpha r_1) .
            \end{multline*}
        \end{enumerate}
    \end{enumerate}
\end{corollary}

\begin{proof}
    Proof is similar to that of Corollary \ref{cor1.1}.
\end{proof}

Taking $G_2$ as $K_{p,q}$, we obtain the $A_{\alpha}$-eigenvalues of $G_1 \underline{\vee} G_2$ in the next corollary.

\begin{corollary}\label{cor2.2}
    Let $G_1$ be an $r_1$-regular graph on $n_1$ vertices and $m_1$ edges. Let $p,q \geq 1$ be integers and $G_2 = K_{p,q}$.
   \begin{enumerate}[label={\upshape \arabic*.}]
        \item If $r_1 =1$, then for each $\alpha \in [0,1]$, the $A_{\alpha}$-spectrum of $G_1 \underline{\vee} G_2$ consists precisely of:
        \begin{enumerate}[label= {\upshape (\roman*)}]
            \item $\alpha$;
            \item $\alpha (1+p)$, repeated $q-1$ times;
            \item $\alpha (1+q)$, repeated $p-1$ times and
            \item four roots of the equation $F(x)=0$, where
            \begin{multline*}
                F(x) = \big( x^2 -\alpha(3+p+q)x + ({\alpha}^2p +{\alpha}^2q +4\alpha -2) \big) \\
                \cdot \big( x^2 -\alpha (2 + p +q)x +( {\alpha}^2 + {\alpha}^2p + {\alpha}^2q +2\alpha pq -pq) \big)\\
                -(1-\alpha)^2(x-\alpha) \big( (x-\alpha)(p+q) -\alpha(p+q)^2 + 2pq \big).
            \end{multline*}
        \end{enumerate}

        \item If $r_1 \geq 2$, then for each $\alpha \in [0,1]$, the $A_{\alpha}$-spectrum of $G_1 \underline{\vee} G_2$ consists precisely of:
        \begin{enumerate}[label= {\upshape (\roman*)}]
            \item $\alpha(2+p+q)$, repeated $m_1 -n_1$ times;
            \item $\alpha (m_1 +p)$, repeated $q-1$ times;
            \item $\alpha (m_1 +q)$, repeated $p-1$ times;
            \item two roots of the equation $G_i(x) = 0$ for each $i=2,3,\ldots,n_1$, where
            \begin{multline*}
                G_i(x) = x^2 - \alpha (2+r_1+p+q)x + ({\alpha}^2 r_1 p + {\alpha}^2 r_1 q + 3\alpha r_1 - r_1)\\
                -(1-\alpha)  \lambda_i\big(A_{\alpha}(G_1)\big)
            \end{multline*}
            and
            \item four roots of the equation $F(x)=0$, where
            \begin{multline*}
                F(x) = \big( x^2 -\alpha (2+r_1+p+q)x + r_1 ({\alpha}^2p +{\alpha}^2q +4\alpha -2)   \big) \\
                \cdot \big( x^2 -\alpha( 2m_1 + p +q)x +( {\alpha}^2m_1^2 + {\alpha}^2m_1p + {\alpha}^2m_1q +2\alpha pq -pq)  \big)\\
                -m_1(1-\alpha)^2(x-\alpha r_1)\big((x-\alpha m_1)(p+q) -\alpha( p+q )^2 +2pq \big).
            \end{multline*}
        \end{enumerate}
    \end{enumerate}
\end{corollary}

\begin{proof}
    Proof is similar to that of Corollary \ref{cor1.2}.
\end{proof}

Finally, to conclude this section, we provide a construction of new pairs of $A_{\alpha}$-cospectral graphs from a given pair of $A_{\alpha}$-cospectral graphs in the following corollary.

\begin{corollary}\label{cor2.3}
    \begin{enumerate}[label={\upshape \arabic*.}]
        \item Let $G_1$ and $G_2$ be two $A_{\alpha}$-cospectral regular graphs for $\alpha \in [0,1]$, and let $H$ be an arbitrary graph. Then the graphs  $G_1 \underline{\vee} H$ and  $G_2 \underline{\vee} H$ are $A_{\alpha}$-cospectral.
        \item Let $H_1$ and $H_2$ be two $A_{\alpha}$-cospectral graphs with $\Gamma_{A_{\alpha}(H_1)}(x) = \Gamma_{A_{\alpha}(H_2)}(x)$ for $\alpha \in [0,1]$. If $G$ is a regular graph, then the graphs $G\underline{\vee} H_1$ and $G  \underline{\vee} H_2$ are $A_{\alpha}$-cospectral.
    \end{enumerate}
\end{corollary}

\begin{proof}
    Proof is similar to that of Corollary \ref{cor1.3}.
\end{proof}


\section{$A_{\alpha}$-spectrum of $G_1 \langle \textrm{v} \rangle G_2$} \label{sec6}
In this section, we study about the $A_{\alpha}$-spectrum of $G_1 \langle \textrm{v} \rangle G_2$, the $R$-vertex join of the graphs $G_1$ and $G_2$. We start with obtaining the expression of the $A_{\alpha}$-characteristic polynomial of $G_1 \langle \textrm{v} \rangle G_2$, for an $r_1$-regular graph $G_1$ and an arbitrary graph $G_2$, in the following theorem.

\begin{theorem}\label{th3}
    Let $G_1$ be an $r_1$-regular graph on $n_1$ and $m_1$ edges, and $G_2$ be an arbitrary graph on $n_2$ vertices. Let $\Gamma_{A_{\alpha}(G_2)}(x)$ be the $A_{\alpha}(G_2)$-coronal of $G_2$. Then for each $\alpha \in [0,1]$, the $A_{\alpha}$-characteristic polynomial of $G_1 \langle \emph{v} \rangle G_2$ is given by,
    \begin{multline*}
        \psi_{A_{\alpha}(G_1 \langle \emph{v} \rangle G_2)}(x) =\big(x-2\alpha \big)^{m_1 -n_1} \cdot \psi_{A_{\alpha}(G_2)}(x- \alpha n_1)\\
        \cdot \prod_{i=2}^{n_1}\Big(x^2 -\big(2\alpha +\alpha r_1+\alpha n_2 + \lambda_i\big(A_{\alpha}(G_1)\big)\big)x + 2{\alpha}^2n_2 +3\alpha r_1 -r_1 + (3\alpha-1)\lambda_i\big(A_{\alpha}(G_1)\big)\Big)\\
        \cdot \Big( x^2 -(2\alpha +\alpha r_1 +\alpha n_2 +r_1)x + 2{\alpha}^2n_2 +6\alpha r_1 -2r_1 -n_1(1-\alpha)^2(x-2\alpha) \Gamma_{A_{\alpha}(G_2)}(x-\alpha n_1) \Big).
    \end{multline*}
\end{theorem}

\begin{proof}
    The adjacency matrix of $G_1 \langle \textrm{v} \rangle G_2$ is
\begin{equation}\label{eq63}
    A(G_1 \langle \textrm{v} \rangle G_2) =
    \begin{bmatrix}
        A(G_1) & R & J_{n_1 \times n_2}\\
        R^T & 0_{m_1 \times m_1} & 0_{m_1 \times n_2}\\
        J_{n_2 \times n_1} & 0_{n_2 \times m_1} & A(G_2)
    \end{bmatrix},
\end{equation}
where $R$ is the $0-1$ incidence matrix of $G_1$.\\
The degrees of the vertices of the graph $G_1 \langle \textrm{v} \rangle G_2$ are:
\begin{equation*}
    \begin{split}
        d_{G_1 \langle \textrm{v} \rangle G_2}(v_i) & = 2r_1 +n_2, \enskip \textrm{for} \enskip i=1,2, \ldots , n_1;\\
        d_{G_1 \langle \textrm{v} \rangle G_2}(v^\prime_j) & = 2, \enskip \textrm{for} \enskip j=1,2, \ldots , m_1;\\
        d_{G_1 \langle \textrm{v} \rangle G_2}(u_k) & = d_{G_2}(u_k) + n_1, \enskip \textrm{for} \enskip k=1,2, \ldots , n_2.
    \end{split}
\end{equation*}
Therefore, 
\begin{equation}\label{eq65}
    D(G_1 \langle \textrm{v} \rangle G_2) =
    \begin{bmatrix}
        (2r_1+n_2) I_{n_1} & 0_{n_1 \times m_1} & 0_{n_1 \times n_2}\\
        0_{m_1 \times n_1} & 2I_{m_1} & 0_{m_1 \times n_2}\\
        0_{n_2 \times n_1} & 0_{n_2 \times m_1} & D(G_2) + n_1 I_{n_2}
    \end{bmatrix}.
\end{equation}
Using \eqref{eq63} and \eqref{eq65}, we get the $A_{\alpha}$-matrix of $G_1 \langle \textrm{v} \rangle G_2$ as
\begin{equation*}\label{eq66}
    A_{\alpha}(G_1 \langle \textrm{v} \rangle G_2)=
    \begin{bmatrix}
        A_{\alpha}(G_1) +(\alpha r_1 +\alpha n_2)I_{n_1} & (1-\alpha)R & (1-\alpha)J_{n_1 \times n_2}\\
        (1-\alpha)R^T & 2\alpha I_{m_1} & 0_{m_1 \times n_2}\\
        (1-\alpha)J_{n_2 \times n_1} & 0_{n_2 \times m_1} & A_{\alpha}(G_2) + \alpha n_1 I_{n_2}
    \end{bmatrix}.
\end{equation*}
Therefore, the characteristic polynomial of $A_{\alpha}(G_1 \langle \textrm{v} \rangle G_2)$ is
\begin{align}\label{eq67}
    \nonumber &\psi_{A_{\alpha}(G_1 \langle \textrm{v} \rangle G_2)}(x) = \det \Big(xI_{n_1+m_1+n_2}-A_{\alpha}(G_1 \langle \textrm{v} \rangle G_2) \Big)\\
    \nonumber &= \det
    \begin{bmatrix}
        (x-\alpha r_1 -\alpha n_2)I_{n_1}-A_{\alpha}(G_1) & -(1-\alpha)R & -(1-\alpha)J_{n_1 \times n_2}\\
        -(1-\alpha)R^T & (x-2\alpha) I_{m_1} & 0_{m_1 \times n_2}\\
        -(1-\alpha)J_{n_2 \times n_1} & 0_{n_2 \times m_1} & (x-\alpha n_1) I_{n_2} - A_{\alpha}(G_2)
    \end{bmatrix}\\
    &=  \det \Big( (x-\alpha n_1) I_{n_2} - A_{\alpha}(G_2) \Big) \cdot \det S \quad (\text{by Lemma } \ref{lem1}),
\end{align}
where
\begin{multline*}
    S=
    \begin{bmatrix}
        (x-\alpha r_1-\alpha n_2) I_{n_1} - A_{\alpha}(G_1) & -(1-\alpha)R\\
        -(1-\alpha)R^T & (x-2\alpha) I_{m_1}
    \end{bmatrix}\\
    -
    \begin{bmatrix}
        -(1-\alpha)J_{n_1 \times n_2}\\
        0_{m_1 \times n_2}
    \end{bmatrix}
    \Big( (x-\alpha n_1) I_{n_2} - A_{\alpha}(G_2) \Big)^{-1}
    \begin{bmatrix}
        -(1-\alpha)J_{n_2 \times n_1} & 0_{n_2 \times m_1}
    \end{bmatrix}
\end{multline*}

\begin{flalign*}
    =\begin{bmatrix}
        (x-\alpha r_1-\alpha n_2)I_{n_1} -A_{\alpha}(G_1)-(1-\alpha)^2\Gamma_{A_{\alpha}(G_2)}(x-\alpha n_1)J_{n_1 \times n_1} & -(1-\alpha)R\\
        -(1-\alpha)R^T & (x-2\alpha)I_{m_1}
    \end{bmatrix}.   &&
\end{flalign*}
By  Lemma \ref{lem1}, we get
\begin{multline*}
    \det S = \det \Big( (x-2\alpha) I_{m_1}\Big) \\
    \cdot \det \Big( (x-\alpha r_1 -\alpha n_2)I_{n_1} - A_{\alpha}(G_1) - (1-\alpha)^2\Gamma_{A_{\alpha}(G_2)}(x-\alpha n_1)J_{n_1 \times n_1} -\frac{(1-\alpha)^2}{(x-2\alpha)}RR^T \Big)
\end{multline*}
\begin{multline}\label{eq71}
    = (x-2\alpha)^{m_1}  \cdot \det \Big( (x-\alpha r_1 - \alpha n_2) I_{n_1} - A_{\alpha}(G_1) - \frac{(1-\alpha)^2}{(x-2\alpha)}RR^T \Big)\\
    \cdot \Big( 1-  (1-\alpha)^2\Gamma_{A_{\alpha}(G_2)}(x-\alpha n_1) \Gamma_{A_{\alpha}(G_1) + \frac{(1-\alpha)^2}{(x-2\alpha)}RR^T}(x-\alpha r_1-\alpha n_2)  \Big),\\
    \text{by Lemma }\ref{lem3}.
\end{multline}
Now using \eqref{eq1} in \eqref{eq4} for the graph $G_1$, we get
\begin{equation*} \label{eq72}
    RR^T=\frac{1}{(1-\alpha)}\big( A_{\alpha}(G_1) -\alpha r_1 I_{n_1}\big) + r_1I_{n_1}.
\end{equation*}
Thus
\begin{equation}\label{eq73}
    A_{\alpha}(G_1) + \frac{(1-\alpha)^2}{(x-2\alpha)}RR^T = \frac{1}{(x-2\alpha)}\Big((x-3\alpha +1)A_{\alpha}(G_1) + (r_1 -3\alpha r_1 + 2{\alpha}^2r_1)I_{n_1}\Big).
\end{equation}
Again, each row sum of the matrix $A_{\alpha}(G_1) + \frac{(1-\alpha)^2}{(x-2\alpha)} RR^T$ is  $r_1 + \frac{(1-\alpha)^2}{(x-2\alpha)}2r_1$.\\
Therefore from \eqref{eq6}, we have 
\begin{equation}\label{eq74}
    \Gamma_{A_{\alpha}(G_1) + \frac{(1-\alpha)^2}{(x-2\alpha)}RR^T}(x-\alpha r_1 -\alpha n_2) = \frac{n_1}{x-\alpha r_1 -\alpha n_2 -\Big(r_1 + \frac{(1-\alpha)^2}{(x-2\alpha)}2r_1\Big)}.
\end{equation}
Applying \eqref{eq73} and \eqref{eq74} in \eqref{eq71} and simplifying, we get
\begin{multline}\label{eq75}
    \det S = (x-2\alpha)^{m_1-n_1}\\
    \cdot \prod_{i=2}^{n_1} \Big(x^2 -\big(2\alpha +\alpha r_1+\alpha n_2 + \lambda_i\big(A_{\alpha}(G_1)\big)\big)x + 2{\alpha}^2n_2 +3\alpha r_1 -r_1 + (3\alpha-1)\lambda_i\big(A_{\alpha}(G_1)\big)\Big)\\
        \cdot \Big( x^2 -(2\alpha +\alpha r_1 +\alpha n_2 +r_1)x + 2{\alpha}^2n_2 +6\alpha r_1 -2r_1 -n_1(1-\alpha)^2(x-2\alpha) \Gamma_{A_{\alpha}(G_2)}(x-\alpha n_1) \Big).
\end{multline}
Finally, using \eqref{eq75} in \eqref{eq67}, we get the required result.
\end{proof}

Now, in the following corollary, we obtain the $A_{\alpha}$-eigenvalues of $G_1 \langle \text{v} \rangle G_2$ taking $G_2$ as an $r_2$-regular graph.

\begin{corollary}\label{cor3.1}
    Let $G_1$ be an $r_1$-regular graph on $n_1$ vertices and $m_1$ edges, and $G_2$ be an $r_2$-regular graph on $n_2$ vertices.
    \begin{enumerate}[label={\upshape \arabic*.}]
        \item If $r_1 =1$, then for each $\alpha \in [0,1]$, the $A_{\alpha}$-spectrum of $G_1 \langle \emph{v} \rangle G_2$ consists precisely of:
        \begin{enumerate}[label= {\upshape (\roman*)}]
            \item $\alpha (3+n_2) -1$;
            \item $2\alpha + \lambda_i\big( A_{\alpha}(G_2) \big)$, $i=2,3, \ldots, n_2$ and
            \item three roots of the equation $F(x)=0$, where
            \begin{equation*} \label{eq76}
                F(x)=(x-2\alpha -r_2)\big( x^2 -(3\alpha+\alpha n_2 +1)x + 2{\alpha}^2 n_2 +6\alpha -2 \big) - 2n_2(1-\alpha)^2(x-2\alpha) .
            \end{equation*}
        \end{enumerate}

        \item If $r_1 \geq 2$, then for each $\alpha \in [0,1]$, the $A_{\alpha}$-spectrum of $G_1 \langle \emph{v} \rangle G_2$ consists precisely of:
        \begin{enumerate}[label= {\upshape (\roman*)}]
            \item $2\alpha$, repeated $m_1 -n_1$ times;
            \item $\alpha n_1 + \lambda_i\big(A_{\alpha}(G_2)\big)$, $i=2,3, \ldots, n_2$;
            \item two roots of the equation $G_i(x) =0$ for each $i= 2,3,\ldots,n_1$, where
            \begin{multline*}
                G_i(x) = x^2 -\Big(2\alpha + \alpha r_1 +\alpha n_2 +\lambda_i\big(A_{\alpha}(G_1)\big)\Big)x\\
                + 2{\alpha}^2 n_2 +3\alpha r_1 -r_1 + (3\alpha -1)\lambda_i\big(A_{\alpha}(G_1)\big)
            \end{multline*}
            and
            \item three roots of the equation $F(x)=0$, where
            \begin{multline*}
                F(x)=(x-\alpha n_1 -r_2)\Big( x^2 -(2\alpha +\alpha r_1 + \alpha n_2 + r_1)x + (2{\alpha}^2n_2 + 6\alpha r_1 - 2r_1)\Big)
                \\- n_1n_2(1-\alpha)^2(x-2\alpha) .
            \end{multline*}
        \end{enumerate}
    \end{enumerate}
\end{corollary}

Taking $G_2$ as $K_{p,q}$, we obtain the $A_{\alpha}$-eigenvalues of $G_1 \langle \text{v} \rangle G_2$ in the next corollary.

\begin{corollary}\label{cor3.2}
    Let $G_1$ be an $r_1$-regular graph on $n_1$ vertices and $m_1$ edges. Let $p,q \geq 1$ be integers and $G_2 = K_{p,q}$.
   \begin{enumerate}[label={\upshape \arabic*.}]
        \item If $r_1 =1$, then for each $\alpha \in [0,1]$, the $A_{\alpha}$-spectrum of $G_1 \langle \emph{v} \rangle G_2$ consists precisely of:
        \begin{enumerate}[label= {\upshape (\roman*)}]
            \item $\alpha (3+p+q) -1$;
            \item $\alpha (2+p)$, repeated $q-1$ times;
            \item $\alpha (2+q)$, repeated $p-1$ times and
            \item four roots of the equation $F(x)=0$, where
            \begin{multline*}
            F(x) = \big( x^2 -(3\alpha + \alpha p + \alpha q +1)x + 2{\alpha}^2p + 2{\alpha}^2q +6\alpha -2 \big) \\
                \cdot \big( x^2 - (4\alpha +\alpha p +\alpha q)x + 4{\alpha}^2 + 2{\alpha}^2p + 2{\alpha}^2q +2\alpha pq -pq \big)\\
                -2(1-\alpha)^2(x-2\alpha) \big( (x-2\alpha)(p+q) -\alpha(p+q)^2 + 2pq \big).
            \end{multline*}
        \end{enumerate}

        \item If $r_1 \geq 2$, then for each $\alpha \in [0,1]$, the $A_{\alpha}$-spectrum of $G_1 \langle \emph{v} \rangle G_2$ consists precisely of:
        \begin{enumerate}[label= {\upshape (\roman*)}]
            \item $2\alpha$, repeated $m_1 -n_1$ times;
            \item $\alpha (n_1 +p)$, repeated $q-1$ times;
            \item $\alpha (n_1 +q)$, repeated $p-1$ times;
            \item two roots of the equation $G_i(x) = 0$ for each $i=2,3,\ldots,n_1$, where
            \begin{multline*}
                G_i(x) = x^2 - \Big(2\alpha +\alpha r_1 + \alpha p + \alpha q +\lambda_i\big(A_{\alpha}(G_1)\big)\Big)x \\
                +2{\alpha}^2p +2{\alpha}^2q +3\alpha r_1 - r_1 + (3\alpha -1)\lambda_i\big(A_{\alpha}(G_1)\big)
            \end{multline*}
            and
            \item four roots of the equation $F(x)=0$, where
            \begin{multline*}
                F(x) = \Big( x^2 -(2\alpha + \alpha r_1 + \alpha p+ \alpha q +r_1)x + 2{\alpha}^2p + 2{\alpha}^2q +6\alpha r_1 -2r_1 \Big) \\
                \cdot \Big( x^2 -(2\alpha n_1 + \alpha p +\alpha q)x + {\alpha}^2n_1^2 + {\alpha}^2n_1p + {\alpha}^2n_1q +2\alpha pq -pq\Big)\\
                -n_1(1-\alpha)^2(x-2\alpha)\Big((x-\alpha n_1)(p+q) -\alpha( p+q )^2 +2pq \Big).
            \end{multline*}
        \end{enumerate}
    \end{enumerate}
\end{corollary}

Finally, to conclude this section, we provide a construction of new pair of $A_{\alpha}$-cospectral graphs from a given pair of $A_{\alpha}$-cospectral graphs in the following corollary.

\begin{corollary}\label{cor3.3}
    \begin{enumerate}[label={\upshape \arabic*.}]
        \item Let $G_1$ and $G_2$ be two $A_{\alpha}$-cospectral regular graphs for $\alpha \in [0,1]$, and let $H$ be an arbitrary graph. Then the graphs $G_1\langle \emph{v} \rangle H$ and  $G_2 \langle \emph{v} \rangle H$ are $A_{\alpha}$-cospectral.
        \item Let $H_1$ and $H_2$ be two $A_{\alpha}$-cospectral graphs with $\Gamma_{A_{\alpha}(H_1)}(x) = \Gamma_{A_{\alpha}(H_2)}(x)$ for $\alpha \in [0,1]$. If $G$ is a regular graph, then the graphs $G\langle \emph{v} \rangle H_1$ and $G \langle \emph{v} \rangle H_2$ are $A_{\alpha}$-cospectral.
    \end{enumerate}
\end{corollary}


\section{$A_{\alpha}$-spectrum of $G_1 \langle \textrm{e} \rangle G_2$} \label{sec7}
In this section, we study the $A_{\alpha}$-spectrum of $G_1 \langle \textrm{e} \rangle G_2$, the $R$-edge join of the graphs $G_1$ and $G_2$. We start with obtaining the expression of the $A_{\alpha}$-characteristic polynomial of $G_1 \langle \textrm{e} \rangle G_2$, for an $r_1$-regular graph $G_1$ and an arbitrary graph $G_2$, in the following theorem.

\begin{theorem}\label{th4}
    Let $G_1$ be an $r_1$-regular graph on $n_1$ vertices and $m_1$ edges, and $G_2$ be an arbitrary graph on $n_2$ vertices. Let $\Gamma_{A_{\alpha}(G_2)}(x)$ be the $A_{\alpha}(G_2)$-coronal of $G_2$. Then for each $\alpha \in [0,1]$, the $A_{\alpha}$-characteristic polynomial of $G_1 \langle \emph{e} \rangle G_2$ is given by,
    \begin{multline*}
        \psi_{A_{\alpha}(G_1 \langle \emph{e} \rangle G_2)}(x) =\big(x-2\alpha -\alpha n_2\big)^{m_1 -n_1} \cdot \psi_{A_{\alpha}(G_2)}(x- \alpha m_1)\\
        \cdot \prod_{i=2}^{n_1}\Big(x^2 -\big(2\alpha +\alpha r_1+\alpha n_2 + \lambda_i\big(A_{\alpha}(G_1)\big)\big)x + {\alpha}^2 r_1 n_2 +3\alpha r_1 -r_1 - (1- 3\alpha-\alpha n_2)\lambda_i\big(A_{\alpha}(G_1)\big)\Big)\\
        \cdot \Big( x^2 -(2\alpha +\alpha r_1 +\alpha n_2 +r_1)x + {\alpha}^2 r_1 n_2 +6\alpha r_1 -2r_1 + \alpha r_1 n_2 \\-m_1(1-\alpha)^2(x-\alpha r_1 -r_1) \Gamma_{A_{\alpha}(G_2)}(x-\alpha m_1) \Big).
    \end{multline*}
\end{theorem}

\begin{proof}
    The adjacency matrix of $G_1 \langle \textrm{e} \rangle G_2$ is
\begin{equation}\label{eq83}
    A(G_1 \langle \textrm{e} \rangle G_2) =
    \begin{bmatrix}
        A(G_1) & R & 0_{n_1 \times n_2}\\
        R^T & 0_{m_1 \times m_1} & J_{m_1 \times n_2}\\
        0_{n_2 \times n_1} & J_{n_2 \times m_1} & A(G_2)
    \end{bmatrix},
\end{equation}
where $R$ is the $0-1$ incidence matrix of $G_1$.\\
The degrees of the vertices of the graph $G_1 \langle \textrm{e} \rangle G_2$ are:
\begin{equation*}\label{eq84}
    \begin{split}
        d_{G_1 \langle \textrm{e} \rangle G_2}(v_i) & = 2r_1, \enskip \textrm{for} \enskip i=1,2, \ldots , n_1;\\
        d_{G_1 \langle \textrm{e} \rangle G_2}(v^\prime_j) & = 2 + n_2, \enskip \textrm{for} \enskip j=1,2, \ldots , m_1;\\
        d_{G_1 \langle \textrm{e} \rangle G_2}(u_k) & = d_{G_2}(u_k) + m_1, \enskip \textrm{for} \enskip k=1,2, \dots , n_2.
    \end{split}
\end{equation*}
Therefore, we have
\begin{equation}\label{eq85}
    D(G_1 \langle \textrm{e} \rangle G_2) =
    \begin{bmatrix}
        2r_1 I_{n_1} & 0_{n_1 \times m_1} & 0_{n_1 \times n_2}\\
        0_{m_1 \times n_1} & (2+n_2)I_{m_1} & 0_{m_1 \times n_2}\\
        0_{n_2 \times n_1} & 0_{n_2 \times m_1} & D(G_2) + m_1 I_{n_2}
    \end{bmatrix}.
\end{equation}
Using \eqref{eq83} and \eqref{eq85}, we get the $A_{\alpha}$-matrix of $G_1 \langle \textrm{e} \rangle G_2$ as
\begin{equation*}\label{eq86}
    A_{\alpha}(G_1 \langle \textrm{e} \rangle G_2)=
    \begin{bmatrix}
        A_{\alpha}(G_1) +\alpha r_1 I_{n_1} & (1-\alpha)R & 0_{n_1 \times n_2}\\
        (1-\alpha)R^T & \alpha(2+n_2) I_{m_1} & (1-\alpha)J_{m_1 \times n_2}\\
        0_{n_2 \times n_1} & (1-\alpha)J_{n_2 \times m_1} & A_{\alpha}(G_2) + \alpha m_1 I_{n_2}
    \end{bmatrix}.
\end{equation*}
Therefore, the characteristic polynomial of $A_{\alpha}(G_1 \langle \textrm{e} \rangle G_2)$ is
\begin{align}\label{eq87}
    \nonumber &\psi_{A_{\alpha}(G_1 \langle \textrm{e} \rangle G_2)}(x) = \det \Big(xI_{n_1+m_1+n_2}-A_{\alpha}(G_1 \langle \textrm{e} \rangle G_2) \Big)\\
    \nonumber &= \det
    \begin{bmatrix}
        (x-\alpha r_1)I_{n_1}-A_{\alpha}(G_1) & -(1-\alpha)R & 0_{n_1 \times n_2}\\
        -(1-\alpha)R^T & (x-2\alpha -\alpha n_2) I_{m_1} & -(1-\alpha)J_{m_1 \times n_2}\\
        0_{n_2 \times n_1} & -(1-\alpha)J_{n_2 \times m_1} & (x-\alpha m_1) I_{n_2} - A_{\alpha}(G_2)
    \end{bmatrix}\\
    &=  \det \Big( (x-\alpha m_1) I_{n_2} - A_{\alpha}(G_2) \Big) \cdot \det S \quad (\text{by Lemma } \ref{lem1}),
\end{align}
where
\begin{multline*}
    S=
    \begin{bmatrix}
        (x-\alpha r_1) I_{n_1} - A_{\alpha}(G_1) & -(1-\alpha)R\\
        -(1-\alpha)R^T & (x-2\alpha -\alpha n_2) I_{m_1}
    \end{bmatrix}\\
    -
    \begin{bmatrix}
        0_{n_1 \times n_2}\\
        -(1-\alpha)J_{m_1 \times n_2}
    \end{bmatrix}
    \Big( (x-\alpha m_1)I_{n_2} - A_{\alpha}(G_2) \Big)^{-1}
    \begin{bmatrix}
        0_{n_2 \times n_1} & -(1-\alpha)J_{n_2 \times m_1}
    \end{bmatrix}
\end{multline*}

\begin{flalign*}
    =\begin{bmatrix}
        (x-\alpha r_1)I_{n_1} -A_{\alpha}(G_1) & -(1-\alpha)R\\
        -(1-\alpha)R^T & (x-2\alpha -\alpha n_2)I_{m_1} -(1-\alpha)^2\Gamma_{A_{\alpha}(G_2)}(x-\alpha m_1)J_{m_1 \times m_1}
    \end{bmatrix}.   &&
\end{flalign*}
By Lemma \ref{lem1}, we get
\begin{multline}\label{eq90}
    \det S = \det \Big( (x-2\alpha -\alpha n_2) I_{m_1} - (1-\alpha)^2\Gamma_{A_{\alpha}(G_2)}(x-\alpha m_1)J_{m_1 \times m_1} \Big) \\
    \cdot \det \bigg((x-\alpha r_1)I_{n_1} - A_{\alpha}(G_1)\\
    -(1-\alpha)^2R\Big((x-2\alpha -\alpha n_2)I_{m_1} -(1-\alpha)^2 \Gamma_{A_{\alpha}(G_2)}(x-\alpha m_1)J_{m_1 \times m_1} \Big)^{-1}R^T \bigg).
\end{multline}
By Lemma \ref{lem5}, we can write \eqref{eq90} as
\begin{multline*}
    \det S = \big(x-2\alpha -\alpha n_2 \big)^{m_1 -1} \Big(x-2\alpha - \alpha n_2 -m_1 (1-\alpha)^2\Gamma_{A_{\alpha}(G_2)}(x-\alpha m_1)\Big)\\
    \cdot \det \bigg( (x-\alpha r_1)I_{n_1} - A_{\alpha}(G_1) -(1-\alpha)^2 R\Big( \frac{1}{(x-2\alpha -\alpha n_2)}I_{m_1}\\
    +\frac{(1-\alpha)^2\Gamma_{A_{\alpha}(G_2)}(x-\alpha m_1)}{(x-2\alpha -\alpha n_2)(x-2\alpha -\alpha n_2 -m_1(1-\alpha)^2 \Gamma_{A_{\alpha}(G_2)}(x-\alpha m_1))}J_{m_1 \times m_1} \Big)R^T \bigg)
\end{multline*}
\begin{multline*}
    = \big(x-2\alpha -\alpha n_2 \big)^{m_1 -1} \Big(x-2\alpha - \alpha n_2 -m_1 (1-\alpha)^2\Gamma_{A_{\alpha}(G_2)}(x-\alpha m_1)\Big)\\
    \cdot \det \bigg( (x-\alpha r_1)I_{n_1} - A_{\alpha}(G_1)   -\frac{(1-\alpha)^2}{(x-2\alpha -\alpha n_2)}RR^T\\
    -\frac{(1-\alpha)^4\Gamma_{A_{\alpha}(G_2)}(x-\alpha m_1)}{(x-2\alpha -\alpha n_2)\big(x-2\alpha -\alpha n_2 -m_1(1-\alpha)^2 \Gamma_{A_{\alpha}(G_2)}(x-\alpha m_1)\big)} RJ_{m_1 \times m_1}R^T \bigg)
\end{multline*}
\begin{multline*}
    = \big(x-2\alpha -\alpha n_2 \big)^{m_1 -1} \Big(x-2\alpha - \alpha n_2 -m_1 (1-\alpha)^2\Gamma_{A_{\alpha}(G_2)}(x-\alpha m_1)\Big)\\
    \cdot \det \bigg( (x-\alpha r_1)I_{n_1} - A_{\alpha}(G_1) - \frac{(1-\alpha)^2}{(x-2\alpha -\alpha n_2)}RR^T\\
    -\frac{{r_1}^2(1-\alpha)^4\Gamma_{A_{\alpha}(G_2)}(x-\alpha m_1)}{(x-2\alpha -\alpha n_2)\big(x-2\alpha -\alpha n_2 -m_1(1-\alpha)^2 \Gamma_{A_{\alpha}(G_2)}(x-\alpha m_1)\big)} J_{n_1 \times n_1} \bigg)
\end{multline*}
\begin{multline}\label{eq94}
    = \big(x-2\alpha -\alpha n_2 \big)^{m_1 -1} \Big(x-2\alpha - \alpha n_2 -m_1 (1-\alpha)^2\Gamma_{A_{\alpha}(G_2)}(x-\alpha m_1)\Big)\\
    \cdot \det \bigg( (x-\alpha r_1)I_{n_1} - A_{\alpha}(G_1) - \frac{(1-\alpha)^2}{(x-2\alpha -\alpha n_2)}RR^T\bigg)\\
    \cdot \bigg(1 - \frac{{r_1}^2(1-\alpha)^4\Gamma_{A_{\alpha}(G_2)}(x-\alpha m_1)}{(x-2\alpha -\alpha n_2)\big(x-2\alpha -\alpha n_2 -m_1(1-\alpha)^2 \Gamma_{A_{\alpha}(G_2)}(x-\alpha m_1)\big)}\\
    \cdot \Gamma_{A_{\alpha}(G_1) + \frac{(1-\alpha)^2}{(x-2\alpha -\alpha n_2)}RR^T}(x-\alpha r_1) \bigg),
\end{multline}
by Lemma \ref{lem3}.

Again, each row sum of the matrix $A_{\alpha}(G_1) + \frac{(1-\alpha)^2}{(x-2\alpha -\alpha n_2)} RR^T$ is  $r_1 + \frac{(1-\alpha)^2}{(x-2\alpha -\alpha n_2)}2r_1$. Therefore from \eqref{eq6}, we have
\begin{equation*}\label{eq95}
    \Gamma_{A_{\alpha}(G_1) + \frac{(1-\alpha)^2}{(x-2\alpha -\alpha n_2)}RR^T}(x-\alpha r_1) = \frac{n_1}{x-\alpha r_1 -\Big(r_1 + \frac{(1-\alpha)^2}{(x-2\alpha -\alpha n_2)}2r_1\Big)}
\end{equation*}
\begin{equation*}\label{eq96}
    = \frac{n_1(x-2\alpha -\alpha n_2)}{(x-\alpha r_1)(x-2\alpha -\alpha n_2) -r_1(x-2\alpha -\alpha n_2) -2r_1(1-\alpha)^2}
\end{equation*}
and using this, the expression
\begin{multline*}
    \bigg(1 - \frac{{r_1}^2(1-\alpha)^4\Gamma_{A_{\alpha}(G_2)}(x-\alpha m_1)}{(x-2\alpha -\alpha n_2)\big(x-2\alpha -\alpha n_2 -m_1(1-\alpha)^2 \Gamma_{A_{\alpha}(G_2)}(x-\alpha m_1)\big)}\\
    \cdot \Gamma_{A_{\alpha}(G_1) + \frac{(1-\alpha)^2}{(x-2\alpha -\alpha n_2)}RR^T}(x-\alpha r_1) \bigg)
\end{multline*}
becomes
\begin{equation}\label{eq98}
    \frac{f(x)-{r_1}^2n_1(1-\alpha)^4 (x-2\alpha -\alpha n_2)\Gamma_{A_{\alpha}(G_2)}(x-\alpha m_1)}{f(x)},
\end{equation}
where
\begin{multline*}
    f(x)=(x-2\alpha -\alpha n_2)\Big(x-2\alpha -\alpha n_2 -m_1(1-\alpha)^2\Gamma_{A_{\alpha}(G_2)}(x-\alpha m_1)\Big)\\
    \cdot \Big((x-\alpha r_1)(x-2\alpha-\alpha n_2) -r_1(x-6\alpha +2{\alpha}^2-\alpha n_2 +2)\Big).
\end{multline*}
Again,
\begin{flalign*}
    (x-\alpha r_1)I_{n_1} - A_{\alpha}(G_1) - \frac{(1-\alpha)^2}{(x-2\alpha -\alpha n_2)}RR^T &&
\end{flalign*}
\begin{flalign*}
    =(x-\alpha r_1)I_{n_1} - A_{\alpha}(G_1) - \frac{(1-\alpha)^2}{(x-2\alpha -\alpha n_2)}\big( A(G_1) + D(G_1)\big) &&
\end{flalign*}
\begin{flalign*}
    =(x-\alpha r_1)I_{n_1} - A_{\alpha}(G_1) - \frac{(1-\alpha)^2}{(x-2\alpha -\alpha n_2)}\bigg( \frac{A_{\alpha}(G_1) - \alpha D(G_1)}{(1-\alpha)} + D(G_1)\bigg) &&
\end{flalign*}
\begin{flalign*}
    =\bigg(x-\alpha r_1 - \frac{r_1(1-3\alpha +2{\alpha}^2)}{(x-2\alpha -\alpha n_2)}\bigg)I_{n_1} - \frac{(x+ 1-3\alpha -\alpha n_2)}{(x-2\alpha -\alpha n_2)} A_{\alpha}(G_1), \text{ as } D(G_1) = r_1 I_{n_1}.&&
\end{flalign*}
Therefore,
\begin{multline}\label{eq103}
    \det \bigg( (x-\alpha r_1)I_{n_1} - A_{\alpha}(G_1) - \frac{(1-\alpha)^2}{(x-2\alpha -\alpha n_2)}RR^T\bigg)\\
    = \prod_{i=1}^{n_1}\bigg(x-\alpha r_1 - \frac{r_1(1-3\alpha +2{\alpha}^2)}{(x-2\alpha -\alpha n_2)} -\frac{(x+ 1-3\alpha -\alpha n_2)}{(x-2\alpha -\alpha n_2)} \lambda_i\big(A_{\alpha}(G_1)\big)\bigg)\\
    = \frac{1}{(x-2\alpha -\alpha n_2)^{n_1}} \prod_{i=1}^{n_1}\bigg((x-\alpha r_1)(x-2\alpha -\alpha n_2)- r_1(1-3\alpha +2{\alpha}^2) \\
    -(x+ 1-3\alpha -\alpha n_2) \lambda_i\big(A_{\alpha}(G_1)\big)\bigg)\\
    = \frac{1}{(x-2\alpha -\alpha n_2)^{n_1}} \bigg((x-\alpha r_1)(x-2\alpha -\alpha n_2)- r_1(x-6\alpha +2{\alpha}^2 -\alpha n_2 +2)\bigg)\\
    \cdot \prod_{i=2}^{n_1}\bigg((x-\alpha r_1)(x-2\alpha -\alpha n_2)- r_1(1-3\alpha +2{\alpha}^2) -(x+ 1-3\alpha -\alpha n_2) \lambda_i\big(A_{\alpha}(G_1)\big)\bigg).
\end{multline}
Using \eqref{eq98} and \eqref{eq103} in \eqref{eq94},
we get
\begin{multline}\label{eq104}
    \det S = (x-2\alpha -\alpha n_2)^{m_1 -n_1 -1} \cdot F_1(x)\\
    \cdot \prod_{i=2}^{n_1}\bigg((x-\alpha r_1)(x-2\alpha -\alpha n_2)- r_1(1-3\alpha +2{\alpha}^2) -(x+ 1-3\alpha -\alpha n_2) \lambda_i\big(A_{\alpha}(G_1)\big)\bigg),
\end{multline}
where
\begin{multline*}
    F_1(x)
    =\Big(x-2\alpha -\alpha n_2 - m_1(1-\alpha)^2\Gamma_{A_{\alpha}(G_2)}(x-\alpha m_1) \Big) \cdot \Big( (x-\alpha r_1)(x-2\alpha -\alpha n_2) \\
    -r_1(x-6\alpha +2{\alpha}^2 -\alpha n_2 +2) \Big) - {r_1}^2n_1(1-\alpha)^4 \Gamma_{A_{\alpha}(G_2)}(x-\alpha m_1).
\end{multline*}
 Since $G_1$ is an $r_1$-regular graph with $n_1$ vertices and $m_1$ edges, $r_1 n_1 =2m_1$. Using this in the expression of $F_1(x)$ and then simplifying it, we get
\begin{multline*}
    F_1(x) = (x-2\alpha -\alpha n_2) \Big(x^2 -(2\alpha +\alpha n_2 +\alpha r_1 +r_1)x\\
    +{\alpha}^2r_1 n_2 + 6\alpha r_1 +\alpha r_1 n_2 -2r_1 -m_1(1-\alpha)^2(x-\alpha r_1 -r_1)\Gamma_{A_{\alpha}(G_2)}(x-\alpha m_1)\Big).
\end{multline*}
We use this expression in \eqref{eq104}. Using \eqref{eq87}, we get the required result.
\end{proof}

Now, in the following corollary, we obtain the $A_{\alpha}$-eigenvalues of $G_1 \langle \text{e} \rangle G_2$ taking $G_2$ as an $r_2$-regular graph.

\begin{corollary}\label{cor4.1}
    Let $G_1$ be an $r_1$-regular graph with $n_1$ vertices and $m_1$ edges, and $G_2$ be an $r_2$-regular graph on $n_2$ vertices.
    \begin{enumerate}[label={\upshape \arabic*.}]
        \item If $r_1 =1$, then for each $\alpha \in [0,1]$, the $A_{\alpha}$-spectrum of $G_1 \langle \emph{e} \rangle G_2$ consists precisely of:
        \begin{enumerate}[label= {\upshape (\roman*)}]
            \item $3\alpha -1$;
            \item $\alpha + \lambda_i\big( A_{\alpha}(G_2) \big)$, $i=2,3, \ldots, n_2$ and
            \item three roots of the equation $F(x)=0$, where
            \begin{multline*}
                F(x)=(x-\alpha -r_2)\big( x^2 -(3\alpha+\alpha n_2 +1)x + {\alpha}^2 n_2 +6\alpha + \alpha n_2 -2 \big) \\
                - n_2(1-\alpha)^2(x-\alpha -1) .
            \end{multline*}
        \end{enumerate}

        \item If $r_1 \geq 2$, then for each $\alpha \in [0,1]$, the $A_{\alpha}$-spectrum of $G_1 \langle \emph{e} \rangle G_2$ consists precisely of:
        \begin{enumerate}[label= {\upshape (\roman*)}]
            \item $\alpha (2+n_2)$, repeated $m_1 -n_1$ times;
            \item $\alpha m_1 + \lambda_i\big(A_{\alpha}(G_2)\big)$, $i=2,3, \ldots, n_2$;
            \item two roots of the equation $G_i(x) =0$ for each $i= 2,3,\ldots,n_1$, where
            \begin{multline*}
                G_i(x) = x^2 -\Big(2\alpha + \alpha r_1 +\alpha n_2 +\lambda_i\big(A_{\alpha}(G_1)\big)\Big)x\\
                + {\alpha}^2 r_1 n_2 +3\alpha r_1 -r_1 + (3\alpha + \alpha n_2 -1)\lambda_i\big(A_{\alpha}(G_1)\big)
            \end{multline*}
            and
            \item three roots of the equation $F(x)=0$, where
            \begin{multline*}
                F(x)=(x-\alpha m_1 -r_2)\\
                \cdot \Big( x^2 -(2\alpha +\alpha r_1 + \alpha n_2 + r_1)x + ({\alpha}^2 r_1 n_2 + 6\alpha r_1 + \alpha r_1 n_2 - 2r_1)\Big)\\
                -m_1n_2(1-\alpha)^2(x-\alpha r_1 -r_1) .
            \end{multline*}
        \end{enumerate}
    \end{enumerate}
\end{corollary}

Taking $G_2$ as $K_{p,q}$, we obtain the $A_{\alpha}$-eigenvalues of $G_1 \langle \text{e} \rangle G_2$ in the next corollary.

\begin{corollary}\label{cor4.2}
    Let $G_1$ be an $r_1$-regular graph on $n_1$ vertices and $m_1$ edges. Let $p,q \geq 1$ be integers and $G_2 = K_{p,q}$.
   \begin{enumerate}[label={\upshape \arabic*.}]
        \item If $r_1 =1$, then for each $\alpha \in [0,1]$, the $A_{\alpha}$-spectrum of $G_1 \langle \emph{e} \rangle G_2$ consists precisely of:
        \begin{enumerate}[label= {\upshape (\roman*)}]
            \item $3\alpha -1$;
            \item $\alpha (1+p)$, repeated $q-1$ times;
            \item $\alpha (1+q)$, repeated $p-1$ times and
            \item four roots of the equation $F(x)=0$, where
            \begin{multline*}
                F(x) = \Big(x^2 -(2\alpha + \alpha p + \alpha q +\alpha +1)x + {\alpha}^2p + {\alpha}^2q +6\alpha + \alpha p +\alpha q -2\Big) \\
                \cdot \Big( x^2 - (2\alpha +\alpha p +\alpha q)x + {\alpha}^2 + {\alpha}^2p + {\alpha}^2q +2\alpha pq -pq \Big)\\
                -(1-\alpha)^2(x-\alpha -1) \big( px +qx -\alpha p -\alpha q -\alpha p^2 -\alpha q^2 -2\alpha pq + 2pq \big).
            \end{multline*}
        \end{enumerate}

        \item If $r_1 \geq 2$, then for each $\alpha \in [0,1]$, the $A_{\alpha}$-spectrum of $G_1 \langle \emph{e} \rangle G_2$ consists precisely of:
        \begin{enumerate}[label= {\upshape (\roman*)}]
            \item $\alpha(2+p+q)$, repeated $m_1 -n_1$ times;
            \item $\alpha (m_1 +p)$, repeated $q-1$ times;
            \item $\alpha (m_1 +q)$, repeated $p-1$ times;
            \item two roots of the equation $G_i(x) = 0$ for each $i=2,3,\ldots,n_1$, where
            \begin{multline*}
                G_i(x) = x^2 - \Big(2\alpha +\alpha r_1 + \alpha p + \alpha q +\lambda_i\big(A_{\alpha}(G_1)\big)\Big)x \\
                +{\alpha}^2r_1p +{\alpha}^2r_1q +3\alpha r_1 - r_1 + (3\alpha +\alpha p +\alpha q -1)\lambda_i\big(A_{\alpha}(G_1)\big)
            \end{multline*}
            and
            \item four roots of the equation $F(x)=0$, where
            \begin{multline*}
                F(x)= \\
                \Big( x^2 -(2\alpha + \alpha r_1 + \alpha p+ \alpha q +r_1)x + {\alpha}^2r_1p + {\alpha}^2r_1q +6\alpha r_1 +\alpha r_1p + \alpha r_1q - 2r_1 \Big) \\
                \cdot \Big( x^2 -(2\alpha m_1 + \alpha p +\alpha q)x + {\alpha}^2m_1^2 + {\alpha}^2m_1p + {\alpha}^2m_1q +2\alpha pq -pq\Big)\\
                -m_1(1-\alpha)^2(x-\alpha r_1 -r_1)\big(px +qx -\alpha m_1p -\alpha m_1q -{\alpha}p^2 -{\alpha}q^2 -2\alpha pq +2pq \big).
            \end{multline*}
        \end{enumerate}
    \end{enumerate}
\end{corollary}

Finally, to conclude this section, we provide a construction of new pair of $A_{\alpha}$-cospectral graphs from a given pair of $A_{\alpha}$-cospectral graphs in the following corollary.

\begin{corollary}\label{cor4.3}
    \begin{enumerate}[label={\upshape \arabic*.}]
        \item Let $G_1$ and $G_2$ be two $A_{\alpha}$-cospectral regular graphs for $\alpha \in [0,1]$, and let $H$ be an arbitrary graph. Then the graphs $G_1\langle \emph{e} \rangle H$ and $G_2 \langle \emph{e} \rangle H$ are $A_{\alpha}$-cospectral.
        \item Let $H_1$ and $H_2$ be two $A_{\alpha}$-cospectral graphs with $\Gamma_{A_{\alpha}(H_1)}(x) = \Gamma_{A_{\alpha}(H_2)}(x)$ for $\alpha \in [0,1]$. If $G$ is a regular graph, then the graphs $G\langle \emph{e} \rangle H_1$ and $G \langle \emph{e} \rangle H_2$ are $A_{\alpha}$-cospectral.
    \end{enumerate}
\end{corollary}

\textbf{Acknowledgement:} Iswar Mahato and  M. Rajesh Kannan would like to thank  Department of Science and Technology, India, for financial support through the Early Carrier Research Award (ECR/2017/000643).

\bibliographystyle{amsplain}
\bibliography{bibliojoin}
\end{document}